\documentclass[english]{article}
\usepackage[T1]{fontenc}
\usepackage[utf8]{inputenc}
\usepackage{babel}
\usepackage{bbold}
\usepackage{multicol}
\usepackage{mathtools}
\usepackage{amsthm}
\usepackage{stmaryrd}
\usepackage{amsmath}
\usepackage{enumitem}
\usepackage[mathscr]{euscript}
\usepackage{amssymb}
\usepackage{multicol}
\usepackage[a4paper,left=3.5cm,right=3.5cm,top=3cm,bottom=3cm]{geometry}
\usepackage{tcolorbox}
\usepackage{dsfont}
\newtheorem{theorem}{Theorem}[section]
\newtheorem{crl}[theorem]{Corollary}
\newtheorem{prop}[theorem]{Proposition}
\newtheorem{lm}[theorem]{Lemma}
\newtheorem{defi}[theorem]{Definition}
\newtheorem{rmq}[theorem]{Remark}
\newtheorem*{nota}{Notation}
\def\restriction#1#2{\mathchoice
	{\setbox1\hbox{${\displaystyle #1}_{\scriptstyle #2}$}
		\restrictionaux{#1}{#2}}
	{\setbox1\hbox{${\textstyle #1}_{\scriptstyle #2}$}
		\restrictionaux{#1}{#2}}
	{\setbox1\hbox{${\scriptstyle #1}_{\scriptscriptstyle #2}$}
		\restrictionaux{#1}{#2}}
	{\setbox1\hbox{${\scriptscriptstyle #1}_{\scriptscriptstyle #2}$}
		\restrictionaux{#1}{#2}}}
\def\restrictionaux#1#2{{#1\,\smash{\vrule height .8\ht1 depth .85\dp1}}_{\,#2}}

\newcommand{\rr}{\mathbb{R}}
\newcommand{\cc}{\mathbb{C}}

\newcommand{\spn}{\text{Span}}
\newcommand{\nn}{\mathbb{N}}
\newcommand{\zz}{\mathbb{Z}}
\newcommand{\dt}{\mathrm{d}t}

\newcommand{\ph}{\varphi}
\newcommand{\ep}{\varepsilon}

\DeclareMathOperator{\ad}{ad}
\numberwithin{equation}{section}
\title{Quadratic obstructions to Small-Time Local Controllability for the multi-input bilinear \\Schrödinger equation}
\author{Théo Gherdaoui\thanks{Univ Rennes, CNRS, IRMAR - UMR 6625, F-35000 Rennes, France}}
\usepackage[pdfborder={0 0 0}]{hyperref}
	
\RequirePackage[hyperref,backend=bibtex, 
style=numeric,maxnames=99
]{biblatex}
\begin{document}
	
	\maketitle
	\begin{abstract}
		We investigate the small-time local controllability (STLC) near the ground state of a bilinear Schrödinger equation when the linearized system is not controllable. It is well known that, for single-input systems, quadratic terms in the state expansion can then lead to obstructions to the STLC of the nonlinear system. In this work, we extend this phenomenon to the multi-input setting, presenting the first example of multi-input quadratic obstructions for PDEs. Our results build upon our previous study of such obstructions for ODEs and provide a functional framework for analyzing them in the bilinear Schrödinger equation.
	\end{abstract}
	
	\section{Introduction}
	\subsection{Model and problem}
	Let $r\in\nn^*$. We study the following equation
	\begin{equation} \label{schr}
		\left\lbrace \begin{array}{lr}
			i \partial_t \psi(t,x)=- \partial_x^2\psi(t,x)-u(t)\cdot\mu(x)\psi(t,x), &t\in (0,T),x\in(0,1), \\
			\psi(t,0)=\psi(t,1)=0,&t \in (0,T), \\
			\psi(0,x)=\psi_0(x),&x\in(0,1),
		\end{array}\right.
	\end{equation}
	where $\mu:(0,1) \to \rr^r$, $T>0$,  $u:(0,T)\to \rr^r$, $\psi:(0,T)\times (0,1) \to \cc$, $\|\psi_0\|_{L^2(0,1)}=1$, and $(x,y)\in\rr^r\mapsto x\cdot y\in\rr$ is the canonical scalar product on $\rr^r$. We will use the notations $u:=\left(u^1,\cdots,u^r\right)$ and $\mu:=\left(\mu_1,\cdots,\mu_r\right)$ in all the document.
	When well defined, the solution is denoted $\psi$. When required, we will write $\psi(\cdot;u,\psi_0)$ to refer to this solution to emphasize its dependence on the different parameters. 
	\medskip
	
	This equation describes the evolution of the wave function $\psi$ of a quantum particle in a $1$D infinite square potential well $(0,1)$, subjected to $r$ electric fields with amplitudes $u^{\ell}(t)$.
	The functions $\mu_{\ell}$, called dipolar moments, model the interaction between the particle's wave function $\psi$ and the electric fields.
	\medskip
	
	This is a multi-input nonlinear control system: 
	\begin{itemize}
		\item[-] the state is $\psi: (0,T) \to \mathcal{S}$, where $\mathcal{S}$ denotes the $L^2(0,1)$ sphere, 
		\item[-] the controls are $u^{\ell}:(0,T) \to \rr$, they act bilinearly on the state.
	\end{itemize}
	\medskip
	
	The ground state is the particular free -- with $u\equiv 0$ -- trajectory $(t,x)\in(0,T)\times (0,1)\mapsto \sqrt{2}\sin(\pi x)e^{-i \pi^2 t}$. We focus on small-time local controllability around it.
	\subsection{Notations and definitions}
	
	Let $T>0$ and $k\in\nn$. We consider the Sobolev space $H^k((0,T),\rr)$, equipped with its usual norm and $H_0^k((0,T),\rr)$ the adherence of $\mathcal{C}^{\infty}_c((0,T),\rr)$ for this norm. 
	
	Unless specified, we will work with complex valued space-dependent functions. The space $L^2(0,1)$ is endowed with the Hermitian scalar product $\langle\cdot,\cdot\rangle$.
	We define \begin{equation}\label{laplacien}A:=-\partial^2_x \text{ with  domain } D(A):=H^2(0,1)\cap H^1_0(0,1).\end{equation} 
	The eigenvalues and eigenvectors of $A$ are
	\begin{equation}\label{vp}\lambda_j:=(j\pi)^2,\qquad\ph_j:=\sqrt{2}\sin(j\pi\cdot),\qquad j\geq1.\end{equation}
	The family $(\ph_j)_{j\geq1}$ forms an orthonormal basis of $L^2(0,1)$. We finally define the spaces $H^s_{(0)}(0,1):=D(A^{\frac s2})$ for $s\geq 0$, equipped with the norm
	$$\left\|\ph\right\|_{H^s_{(0)}}:=\left(\sum_{j=1}^{+\infty}|j^s\langle\ph,\ph_j\rangle|^2\right)^{1/2}.$$
	With these notations, \eqref{schr} is well-posed in the sense of the following proposition, proved for the single-input case in \cite[Proposition $2$]{beauchard2010local} and adaptable to our case.
	\begin{prop}[Well-posedness]\label{well-posed} Let $T>0$, $\mu\in H^3((0,1),\rr^r)$, $\psi_0\in H^3_{(0)}(0,1)$ and $u\in L^2((0,T),\rr^r)$. There exists a unique mild solution to \eqref{schr}, \textit{i.e.}\ a function $\psi\in\mathcal{C}^0\left([0,T],H_{(0)}^3(0,1)\right)$ such that \textit{(s.t.)}\ in $H^{3}_{(0)}(0,1)$, for every $t\in[0,T]$,
		\begin{equation}\label{duhamel}\psi(t)=e^{-iAt}\psi_0+i\int_0^te^{-iA(t-s)}((u(s)\cdot \mu)\psi(s))\mathrm{d}s.\end{equation}
	\end{prop}
	For $j\geq 1$, we denote by
	$\psi_j:=\psi(\cdot;0,\ph_j)$ 
	the solution to the free Schrödinger equation  -- \textit{i.e.}\ \eqref{schr} $u\equiv 0$ -- with initial data $\ph_j$ at time $t=0$. The solution $\psi_1$ is called the ground state. We can now define the notion of controllability we are working on.
	\begin{defi}\label{defistlc} Let $E_T$ be a family of normed vector spaces of functions on $[0,T]$ for $T>0$. Equation \eqref{schr} is $E$-STLC in $H^3_{(0)}(0,1)$ around the ground state if for every $T, \varepsilon>0$, there exists $\delta>0$ \textit{s.t.}\ for every $\psi_f \in H^3_{(0)}(0,1)\cap \mathcal{S}$ with $\|\psi_f - \psi_1(T)\|_{H^3}<\delta$, there exists $u\in\left(E_T\right)^r\cap L^2((0,T),\rr^r)$ \textit{s.t.}\ $\left\|u\right\|_{E_T}\leq\varepsilon$ and $\psi(T;u,\varphi_1)=\psi_f$.
	\end{defi}
	Under appropriate assumptions on the dipolar moments $\mu_{\ell}$, the STLC around the ground state of \eqref{schr} can be proved by applying the linear test. 
	This is the purpose of the following statement, that can be proved by adapting with several controls the strategy developed in \cite{beauchard2010local} with one control.
	\begin{prop}\label{lintest}
		Let $\mu\in H^3((0,1),\rr^r)$ be \textit{s.t.}\
		\begin{equation} \label{hyp_linearisé_controlable}
			\exists C>0, \ \forall j \in \nn^*, \
			\sum_{\ell\in\llbracket 1,r\rrbracket}|\langle \mu_{\ell} \varphi_1 , \varphi_j \rangle| \geq \frac{C}{j^3}.
		\end{equation}
		Then,  \eqref{schr} is $L^2$-STLC in
		$H^3_{(0)}(0,1)$ around the ground state.
	\end{prop} 
	Assumption (\ref{hyp_linearisé_controlable}) guarantees that the linearized system around the ground state -- see \eqref{schro:lin} -- is small-time controllable in $H^3_{(0)}(0,1)$ and controls in $L^2$. Indeed, we use the theory of moment problems based on the gap condition of the eigenvalues -- see \eqref{vp}, and the corresponding moment problems are defined in $l^2(\nn^*)$ thanks to inequality \eqref{hyp_linearisé_controlable}.
	Moreover, one can prove that the following end-point map is of class $\mathcal{C}^1$
	$$\Theta_T:u\in L^2((0,T),\rr^r) \mapsto \psi(T;u,\varphi_1) \in H^3_{(0)}(0,1)\cap \mathcal{S}.$$
	By applying the inverse mapping theorem, we obtain the $L^2$-STLC of system \eqref{schr}.
	\subsection{Main result}\label{dipolar}
	
	For all the document, $k,K,r\in\nn^*$ are fixed integers. To ensure that \eqref{schr} is well-posed in $H^3_{(0)}(0,1)$ -- see Proposition \ref{well-posed} -- we assume that
	$$\mathbf{(H)_{reg}}: \quad\mu\in H^3((0,1),\rr^r).$$
	In this article, we study the case where the linearized system around the ground state -- see \eqref{schro:lin} -- is not controllable. We then assume 
	$$\mathbf{(H)_{lin}}:\quad \langle\mu_{\ell}\ph_1,\ph_K\rangle=0,\quad 1\leq\ell\leq r.$$
	In this case, the linearized system misses one direction $\langle \psi(t;u,\ph_1), \varphi_K \rangle \in \cc$,
	\textit{i.e.}\ the first order term in the Taylor expansion at $0$ of the map $u \mapsto \langle \psi(t;u,\varphi_1),\varphi_K \rangle$ vanishes.
	To prove an obstruction to STLC, we use a power series expansion of order $2$, \textit{i.e.}\ 
	the second order term in the Taylor expansion at $0$  of $u \mapsto \langle \psi(t;u,\ph_1), \varphi_K \rangle$. We make assumptions about the quadratic expansion of the solution.
	\begin{defi}\label{defcj} For all $1\leq\ell,L\leq r$, we define the following sequences $$c^{\ell,L}:=\left(c^{\ell,L}_j\right)_{j\geq 1}:=\left(\left\langle\mu_{\ell}\ph_K,\ph_j\right\rangle\left\langle \mu_L\ph_1,\ph_j\right\rangle\right)_{j\geq 1}.$$
		We will denote $c^{\ell}:=c^{\ell,\ell}$ and $c_j^{\ell}:=c_j^{\ell,\ell}$ for $j\geq 1$ and $1\leq\ell\leq r$.
	\end{defi}
	Quadratic hypotheses are formulated as properties on series. To ensure convergence, we assume
	\begin{equation*}\mathbf{(H)_{conv}}: \quad \displaystyle\sum_{j=1}^{+\infty}\left|c_j^{\ell,L}\right|j^{4k}<+\infty, \quad 1\leq\ell,L\leq r.\end{equation*}
	\begin{rmq}\label{rmqassympto} By Cauchy--Schwarz's inequality, $\mathbf{(H)_{conv}}$ holds when $\mu_{\ell}\ph_K$ and $\mu_L\ph_1$ belong to $H^{2k}_{(0)}(0,1)$, for every $1\leq\ell,L\leq r$. This is in particular the case when $\mu \in H^{2k}((0,1),\rr^r)$ with vanishing odd derivatives, or $\mu\in\mathcal{C}^{\infty}_c((0,1),\rr^r)$.
	\end{rmq}
	\begin{defi}[Quadratic brackets]\label{defquadbra}Assume that $\mathbf{(H)_{conv}}$ holds. For every $0\leq p\leq 2k-1$ and $1\leq \ell\leq L\leq r$, we define the real number
		$$\gamma_p^{\ell,L}:=\sum_{j=1}^{+\infty}\left((\lambda_K-\lambda_j)^{\lfloor\frac{p+1}{2}\rfloor}(\lambda_j-\lambda_1)^{\lfloor\frac{p}{2}\rfloor}c_j^{\ell,L}-(\lambda_K-\lambda_j)^{\lfloor\frac{p}{2}\rfloor}(\lambda_j-\lambda_1)^{\lfloor\frac{p+1}{2}\rfloor}c_j^{L,\ell}\right).$$
		The series converge thanks to $\mathbf{(H)_{conv}}$.
	\end{defi}
	\begin{nota}
		When $\ell=L$, we will note $\gamma_p^{\ell}$ instead of $\gamma_p^{\ell,\ell}$. 
	\end{nota}
	\begin{rmq}
		For every $1\leq\ell\leq r$, $0\leq p\leq k-1$, one has $\gamma^{\ell}_{2p}=0$.
	\end{rmq}
	We are now able to state the hypotheses concerning the quadratic terms
	\begin{equation*}\begin{gathered}
			\mathbf{(H)_{null}}: \quad \forall1\leq \ell\leq L\leq r, \ \forall 0\leq p\leq 2k-2, \quad \gamma_p^{\ell,L}=0,\\\mathbf{(H)_{pos}}:  \quad q:(a_1,\cdots,a_r)\in\rr^r\mapsto \displaystyle\sum_{\ell\in\llbracket 1,r\rrbracket}\gamma_{2k-1}^{\ell}\frac{a_{\ell}^2}{2}+\sum_{1\leq \ell<L\leq r}\hspace{-0.2 cm}\gamma_{2k-1}^{\ell,L}a_{\ell}a_L\in\rr\\
			\text{ is a definite quadratic form on }\rr^r.\end{gathered}\end{equation*}
	These assumptions can be interpreted in terms of Lie brackets -- see Section \ref{lienode}.
	\begin{rmq}\label{rmqcoer} As $q(1,0,\cdots,0)=\gamma_{2k-1}^1$, $\mathbf{(H)_{pos}}$ ensures that $\text{sgn}(\gamma_{2k-1}^1)q$ is a positive definite quadratic form on $\rr^r$.
	\end{rmq}
	\begin{rmq}\label{rmqliner=2}When $r=1$, $\mathbf{(H)_{pos}}$ holds \textit{iff} $\gamma_{2k-1}^{1}\neq 0$.
		When $r=2$, $\mathbf{(H)_{pos}}$ holds \textit{iff} $\left(\gamma_{2k-1}^{1,2}\right)^2<\gamma_{2k-1}^{1}\gamma_{2k-1}^{2}$.
	\end{rmq}
	The main result of this article is the following theorem.
	\begin{theorem}\label{maintheorem}
		Let $k,K,r\in \nn^*$,
		$\mu$ satisfying $\mathbf{(H)_{reg}}$, $\mathbf{(H)_{lin}}$, $\mathbf{(H)_{conv}}$, $\mathbf{(H)_{null}}$ and $\mathbf{(H)_{pos}}$.
		If $k\geq 2$ (\textit{resp.}\ if $k=1$), the multi-input bilinear Schrödinger equation \eqref{schr} is not $H^{2k-3}$-STLC (\textit{resp.}\ $W^{-1,\infty}$-STLC) in $H^3_{(0)}(0,1)$ around the ground state.
	\end{theorem}
	
	Explicit examples of functions $\mu$ are given in Remark \ref{rmqmuexplicite} for the case where $r=2$, $k=1$ and $K\geq 9$ an odd integer. The existence of such functions $\mu$ is proved in Theorem \ref{theoremexistence} for arbitrary values of $r,k$ and $K$.
	
	More generally, this result shows that, under the hypotheses of Theorem \ref{maintheorem} with $k=1$, \eqref{schr} is not small-state STLC in $H^3_{(0)}(0,1)$ around the ground state -- see Theorem \ref{theoremsmallstate}. In particular, this framework demonstrates that an obstruction can be established even when the control $u$ is not small, but merely belongs to $L^2$.
	\begin{rmq} When $r=1$, we then obtain \cite[Theorem $1.3$]{bournissou2021quad}.
	\end{rmq}
	\subsection{Heuristic}\label{heurist}
	To simplify the notations, we will use the following definition. 
	
	\begin{defi} For $T>0$ and $f\in L^1((0,T),\rr)$, one defines the iterated primitives $f_n$ vanishing at $0$ by induction as
		\begin{equation}\label{primitive}f_0:=f \quad  \text{and} \quad \forall n\in\nn, \  \forall t\in [0,T], \quad f_{n+1}(t):=\int_0^tf_n(\tau)\mathrm{d}\tau.\end{equation}
		When $f=(f^1,\cdots,f^r)\in L^1((0,T),\rr^r)$, $f_n$ will denote $(f_n^1,\cdots,f_n^r)$. 
	\end{defi}
	
	\begin{enumerate}
		\item[\textbullet] Toy model for $k=1$. The essence of Theorem \ref{maintheorem} can be found in the following example of a finite-dimensional control-affine system
		\begin{equation}\label{jouetex}\left\lbrace\begin{array}{rcl}x_1'&=&u\\x_2'&=&x_1\\x_3'&=&v\\x_4'&=&\left(\left(x_1\right)^2+\frac{1}{2}x_1x_3+2\left(x_3\right)^2\right)-\left(x_2\right)^2-v\left(x_1\right)^2-2vx_3\end{array}\right.,\end{equation}
		with initial data $x(0)=0$. In this two-input control-affine system, the controls are denoted as $u$ and $v$ for readability. The notation will be different in the following sections.
		Let $T>0$. Explicit integration from $0$ yields
		\begin{align}\label{ghu}x_4(T)&=\int_0^T\left(u_1\right)^2+\frac{1}{2}u_1v_1+2\left(v_1\right)^2-\hspace{-0.1 cm}\int_0^T\hspace{-0.1 cm}\left(\int_0^tu_1\right)^2\hspace{-0.2 cm}\dt-\int_0^Tv\left(u_1\right)^2-v_1(T)^2\\&\geq \left(\frac{3}{4}-\frac{T^2}{2}-\left\|(u,v)\right\|_{L^{\infty}(0,T)}\right)\left\|(u_1,v_1)\right\|_{L^2(0,T)}^2-v_1(T)^2\nonumber\\&\geq C\left\|(u_1,v_1)\right\|_{L^2(0,T)}^2-v_1(T)^2,\nonumber\end{align}for any $C\in(0,\frac 34)$, for small enough times and controls in $L^{\infty}$. As $x_3(T)=v_1(T)$, all final state in $\{(x_1,\cdots,x_4)\in\rr^4; \ x_4+\left(x_3\right)^2<0\}$ is not reachable from $0$. Consequently, \eqref{jouetex} is not $L^{\infty}$-STLC around $0$ - see Definition \ref{stlc}. For $k=1$, the dominants terms in the expansion of $(u,v)\mapsto\left\langle\psi(T;(u,v),\ph_1),\ph_Ke^{-i\lambda_1T}\right\rangle$ at $(u,v)=0$ are similar to the ones in \eqref{ghu}.
		\item[\textbullet] Toy model for $k\in\nn^*$ arbitrary. Heuristically, it is as if the following terms are the dominant part of the quadratic expansion of the map $(u,v)\mapsto\left\langle\psi(T;(u,v),\ph_1),\ph_Ke^{-i\lambda_1T}\right\rangle$ at $(u,v)=0$ \begin{equation}\label{heurisart3termdom}\sum_{\ell\in\llbracket 1,r\rrbracket}\sum_{p=1}^{k}i^{2p-1}\gamma_{2p-1}^{\ell} \int_0^T \frac{\left(u_p^{\ell}\right)^2}{2}+ \sum_{1\leq \ell<L\leq r}\sum_{p=0}^{2k-1} i^p\gamma_p^{\ell,L}\int_0^T u^{\ell}_{\lfloor\frac{p}{2}\rfloor+1}u^L_{\lfloor\frac{p+1}{2}\rfloor} .\end{equation}
		The cancellation assumption $\mathbf{(H)_{null}}$ reduces this sum to the simpler expression
		\begin{equation*}
			i(-1)^{k+1}\sum_{\ell\in\llbracket 1,r\rrbracket}\gamma_{2k-1}^{\ell}\int_0^T\frac{\left(u_k^{\ell}\right)^2}{2}+i(-1)^{k+1}\sum_{1\leq \ell<L\leq r}\gamma_{2k-1}^{\ell,L} \int_0^T u_k^{\ell}u_k^L.
		\end{equation*}
		Finally, $\mathbf{(H)_{lin}}$ and $\mathbf{(H)_{pos}}$ allows use to take advantage of a signed term as
		\begin{equation*}(-1)^{k+1}\text{sgn}(\gamma_{2k-1}^1)\Im\left\langle \psi(T;u,\ph_1),\ph_Ke^{-i\lambda_1T}\right\rangle\simeq \text{sgn}(\gamma_{2k-1}^1)\int_0^Tq(u_k)\geq 0.\end{equation*}
	\end{enumerate}
	\subsection{State of the art}\label{stateofart}
	\subsubsection{Topological obstructions to exact controllability}
	In \cite{doi:10.1137/0320042}, Ball, Marsden and Slemrod proved obstructions to local exact controllability of linear PDEs with bilinear controls. Turicini adapted this statement to Schrödinger equations in \cite{turinici:hal-00536518} :
	if the functions $\mu_{\ell}$ define bounded operators on $H^s_{(0)}(0,1)$, then system \eqref{schr} is not exactly controllable in $\mathcal{S}\cap H^s_{(0)}(0,1)$ with control $u\in L^{1+}_{\text{loc}}\left(\rr^+,\rr^r\right)$.  Indeed, the reachable set has empty interior (more precisely, it is contained in a countable union of compact subsets). The case of $L^1_{\text{loc}}$-controls was incorporated in \cite{BOUSSAID2020108412} and extensions to nonlinear equations were proved in \cite{doi:10.1137/18M1215207,chambrion:hal-01901819}. After the important work \cite{doi:10.1137/0320042}, different notions of controllability were studied for the single-input bilinear Schrödinger equation such as
	\begin{itemize}
		\item[-] \textit{exact controllability} in more regular spaces, on which the $\mu_{\ell}$ do not define bounded operators -- see \cite{beauchard2010local,bournissou2022smalltime,gherdaoui},
		\item[-] \textit{approximate controllability}.
	\end{itemize}
	\subsubsection{Exact controllability in more regular spaces, by linear test} 
	For the single-input bilinear Schrödinger equation, local exact controllability was first proved in \cite{BEAUCHARD2005851,BEAUCHARD2006328} with Nash-Moser techniques, to handle an apparent derivative loss and then in \cite{beauchard2010local} with a classical inverse mapping theorem, 
	thanks to a regularizing effect. By grafting other ingredients onto this core strategy, exact controllability in regular spaces was proved for different models
	in \cite{bournissou2021local,duca2022local,AIHPC_2014__31_3_501_0,morancey2013globalexactcontrollability1d,article}.
	\subsubsection{Power series expansion of order 2 or 3} 
	First, let us look at the case of \textbf{single-input systems}. In \cite{Cerpa2009}, the authors use a power series expansion of order $2$ to recover a lost direction in large time for the KdV equation. If the order $2$ vanishes, a power series of order $3$ can be used to recover the STLC -- see \cite{Coron2004}. This method is also used for other equations, for example in \cite{beauchard2025obstructionsmalltimelocalcontrollability,beauchard2013local} for the single-input bilinear Schrödinger equation. If the order $2$ doesn't cancel out, but the term of order $3$ is strong enough, this expansion can also give the STLC -- see \cite{bournissou2022smalltime}.

	In the context of \textbf{multi-input systems}, we use a power series expansion of order $2$ in \cite{gherdaoui} to recover in small-time a direction lost at the linear order. 
	
	\subsubsection{Quadratic obstructions to STLC} 
	In \cite{CORON2006103}, Coron denied the $L^{\infty}$-STLC result for the single-input bilinear Schrödinger equation, with a particular dipolar moment $\mu_{\ell}$, thanks to a drift. In \cite{beauchard2013local}, Beauchard and Morancey gave general assumptions on $\mu_1$ to deny $L^{\infty}$-STLC for the same system. In \cite{bournissou2021quad}, Bournissou proved that this drift also occurs with small control in $W^{-1,\infty}$ and studied all quadratic drift orders. In \cite{beauchard2025obstructionsmalltimelocalcontrollability}, Beauchard, Marbach and Perrin studied the single-input bilinear Schrödinger equation with Neumann boundary conditions, using a power series of order $2$. Quadratic terms have also been used to create obstructions to the controllability of other single-input systems, for example in \cite{beauchard2017quadratic} for ODEs, in \cite{Fr_d_ric_Marbach_2018} for the Burgers' equation, \cite{BEAUCHARD202022} for the heat equation and \cite{coron2020smalltime,nguyen2023localcontrollabilitykortewegdevries} for KdV.
	
	In \cite{gherdaoui2}, we prove quadratic obstructions for multi-input control-affine systems in the finite-dimensional case. In this article, our strategy is to adapt \cite[Theorem $1.22$]{gherdaoui2} for \eqref{schr}. The strong analogy between these statements is developed in Section \ref{lienode}.
	\subsubsection{Approximate controllability} 
	The first results of global approximate controllability of bilinear Schrödinger equations 
	were obtained in large time -- see \cite{Boscain_2012,Chambrion2008ControllabilityOT,ervedozapuel,nersesyan,Sigalotti2009GenericCP}.
	For particular systems, a large time is indeed necessary for the approximate controllability -- see \cite{beauchard2014minimaltimebilinearcontrol,articletkc}. Small-time approximate controllability between eigenstates for Schrödinger equations on the torus is proved in \cite{duca2021bilinearcontrolgrowthsobolev}.
	Related results have been subsequently established in \cite{boscain:hal-04496433,chambrion2022smalltime,duca2024smalltimecontrollabilitynonlinearschrodinger}. The first examples of small-time globally approximately controllable bilinear Schrödinger equations are given in \cite{beauchard2024examplesmalltimegloballyapproximately,beauchard2025smalltimeapproximatecontrollabilitybilinear}.
	\section{Proof of the main theorem}\label{proofmain}
	One recalls that $\lambda_j$ and $\ph_j$ are defined in \eqref{vp}. In all the document, we will note \begin{equation}\label{omeganu}\forall j\in\nn^*, \quad \omega_j:=\lambda_j-\lambda_1\text{ and } \nu_j:=\lambda_K-\lambda_j.\end{equation}
	\subsection{Expansion of the solution}
	We are going to make an asymptotic expansion of the solution to \eqref{schr}. Let $u\in L^2((0,T),\rr^r)$ be a fixed control.
	\begin{enumerate}
		\item[\textbullet ]The term of order $0$ is $\psi_1$. One has, for every $t\in[0,T]$ \begin{equation}\label{nullzero}\Im\left\langle \psi_1(t),\ph_Ke^{-i\lambda_1t}\right\rangle=\Im\left(\delta_{1,K}\right)=0.\end{equation}
		\item[\textbullet] The first-order term $\Psi\in\mathcal{C}^0\left([0,T],H^3_{(0)}(0,1)\right)$ is the solution to the linearized system of \eqref{schr} around the free trajectory $(\psi_1,u\equiv 0)$, \textit{i.e.}\
		\begin{equation}\label{schro:lin}\left\lbrace\begin{array}{lr} i\partial_t\Psi(t,x)=-\partial_x^2\Psi(t,x)-u(t)\cdot\mu(x)\psi_1(t,x), & t\in(0,T),x\in(0,1),\\
				\Psi(t,0)=\Psi(t,1)=0, & t\in(0,T),\\
				\Psi(0,x)=0, &x\in(0,1).\end{array}\right.
		\end{equation}
		Using \eqref{duhamel}, the solution is given by: $\forall t\in [0,T]$, 
		\begin{equation}\label{schro:lin:exp} \Psi(t)=i\sum_{\ell\in\llbracket 1,r\rrbracket}\sum_{j=1}^{+\infty}\left(\left\langle\mu_{\ell}\ph_1,\ph_j\right\rangle\int_0^tu^{\ell}(\tau)e^{i\omega_j\tau}\mathrm{d}\tau\right)\psi_j(t).
		\end{equation}  
		Using $\mathbf{(H)_{lin}}$, one has, for every $t\in[0,T]$,
		\begin{equation}\label{nullin}\left\langle\psi(t),\ph_Ke^{-i\lambda_1t}\right\rangle=0.\end{equation}
		Thus, the linearized system \eqref{schro:lin} is not controllable. 
		\item[\textbullet] The second-order term $\xi\in\mathcal{C}^0\left([0,T],H^3_{(0)}(0,1)\right)$ is the solution to the following system
		\begin{equation}\label{schr:quad:exp}\left\lbrace\begin{array}{lr} i\partial_t\xi(t,x)=-\partial_x^2\xi(t,x)-u(t)\cdot\mu(x)\Psi(t,x),& t\in(0,T),x\in(0,1),\\
				\xi(t,0)=\xi(t,1)=0, & t\in(0,T),\\
				\xi(0,x)=0, &x\in(0,1).\end{array}\right.
		\end{equation}
		Using once again \eqref{duhamel}, the solution is given by: $\forall t\in [0,T]$, 
		\begin{equation}\label{schr:quad:exp:duh} \xi(t)=i\sum_{\ell\in\llbracket 1,r\rrbracket}\sum_{j=1}^{+\infty}\int_0^t\left(\left\langle\mu_{\ell}\Psi(\tau),\ph_j\right\rangle u^{\ell}(\tau)e^{i\lambda_j\tau}\mathrm{d}\tau\right)\psi_j(t).\end{equation}
	\end{enumerate}
	Using \eqref{nullzero} and \eqref{nullin}, one finally has
	\begin{equation}\label{err-heuristique}\Im\left\langle\psi(T;u,\ph_1),\ph_Ke^{-i\lambda_1T}\right\rangle\simeq\Im\left\langle\xi(T),\ph_Ke^{-i\lambda_1T}\right\rangle.\end{equation}
	For $1\leq\ell,L\leq r$, one defines
	$	H_{\ell,L}:(t,s)\in[0,T]^2\mapsto-e^{-i\omega_KT}\displaystyle\sum_{j=1}^{+\infty}c_j^{\ell,L} e^{i(\nu_jt+\omega_js)}\in\cc.$
	\begin{rmq}
		The kernels introduced in this article differ by a phase $e^{-i\omega_KT}$ compared with those introduced in \cite{gherdaoui} in the case where $r=2$.
	\end{rmq}
	We finally use the notation, for $f,g\in L^2((0,T),\rr)$, $1\leq \ell,L\leq r$,
	\begin{equation*}\mathcal{F}_T^{\ell,L}(f,g):=\int_0^Tf(t)\left(\int_0^tH_{\ell,L}(t,\tau)g(\tau)\mathrm{d}\tau\right)\dt.
	\end{equation*}
	With these notations, \eqref{schro:lin:exp} and \eqref{schr:quad:exp:duh},
	\begin{align}
		\left\langle\xi(T),\ph_Ke^{-i\lambda_1T}\right\rangle&=-e^{-i\omega_KT}\sum_{j=1}^{+\infty}\sum_{l,L=1}^rc_j^{\ell,L}\int_0^Tu^{\ell}(t)e^{i\nu_jt}\left(\int_0^tu^L(s)e^{i\omega_js}\mathrm{d}s\right)\mathrm{d}t,\notag
		\\&=\sum_{\ell\in\llbracket 1,r\rrbracket}\mathcal{F}_T^{\ell,\ell}(u^{\ell},u^{\ell})+\sum_{1\leq \ell<L\leq r}\left(\mathcal{F}_T^{\ell,L}(u^{\ell},u^L)+\mathcal{F}_T^{L,\ell}(u^L,u^{\ell})\right).\label{schro:quad}
	\end{align}
	\subsection{Error estimates}\label{section2.2}
	
	We recall that the iterated primitives of a function are defined in \eqref{primitive}. We use sharp error estimates to prove that the remainder term of the expansion \eqref{err-heuristique} can be neglected compared to the drift $\left\|u_k\right\|_{L^2(0,T)}^2$.  A rough error estimate -- see for example \cite[Lemma $3.7$]{gherdaoui} -- involves the $L^2$-norm of the controls. As shown in \cite{bournissou2021quad,bournissou2022smalltime} in the single-input case, one can prove a better estimate involving the $L^2$-norm of the primitive $u_1$ of the controls by introducing the new state
	\begin{equation}\label{lienpsipsitilde}\widetilde{\psi}((t,x);u,\ph_1):=\psi((t,x);u,\ph_1)e^{-iu_1(t)\cdot\mu(x)}.\end{equation}
	which is the weak solution to
	\begin{equation}\label{syst-aux}\left\lbrace\begin{array}{lr} i\partial_t\widetilde{\psi}=-\partial_x^2\widetilde{\psi}-iu_1(t)\cdot\left[\mu''(x)+2\mu'(x)\partial_x\right]\widetilde{\psi}+\left(u_1(t)\cdot\mu'(x)\right)^2\widetilde{\psi},& (0,T)\times(0,1),\\
			\widetilde{\psi}(t,0)=\widetilde{\psi}(t,1)=0,& (0,T),\\
			\widetilde{\psi}(0,x)=\ph_1(x),&(0,1).\end{array}\right.
	\end{equation}
	
	The well-posedness of the auxiliary system \eqref{syst-aux} is stated in \cite[Proposition $4.2$]{bournissou2021quad} in the case where $r=1$ and can be adapted, without any additional difficulty, to our setting. We want to study the linear and quadratic expansions of \eqref{syst-aux} around the ground state. Linearizing \eqref{syst-aux}, the first-order term $\widetilde{\Psi}$ is given by
	\begin{equation}
		\label{linaux}
		\widetilde{\Psi}(t,x)=\Psi(t,x) - iu_1(t)\cdot\mu(x) \psi_1(t,x),
	\end{equation}
	where $\Psi$ is the solution to \eqref{schro:lin}. Thus, $\widetilde{\Psi}\in \mathcal{C}^0\left([0,T], H^3 \cap H^1_0(0,1)\right)$ is a weak solution to
	\begin{equation*}   \left\{ 
		\begin{array}{lr}
			i \partial_t \widetilde{\Psi} = - \partial^2_x \widetilde{\Psi} -iu_1(t)\cdot\left[\mu''(x)+2\mu'(x)\partial_x\right]\psi_1,& (0,T)\times(0,1), \\
			\widetilde{\Psi}(t,0) = \widetilde{\Psi}(t,1)=0,& (0,T),\\
			\widetilde{\Psi}(0,x)=0,& (0,1).
		\end{array}\right. \end{equation*}
	Doing an expansion of order $2$ of \eqref{syst-aux}, the second-order term $\widetilde{\xi}$ is given by
	\begin{equation}
		\label{defauxqua}
		\widetilde{\xi}(t,x)= \xi(t,x) - iu_1(t)\cdot\mu(x) \Psi(t,x) - \frac{\left(u_1(t)\cdot\mu(x)\right)^2}{2} \psi_1(t,x),
	\end{equation}
	where $\xi$ is the solution to \eqref{schr:quad:exp}. Then, $\widetilde{\xi}\in \mathcal{C}^0\left([0,T], H^3 \cap H^1_0(0,1)\right)$ is a weak solution to
	\begin{equation*}   \left\{ \begin{array}{lr}
			i \partial_t \widetilde{\xi} = - \partial^2_x \widetilde{\xi}-iu_1(t)\cdot\left[\mu''(x)+2\mu'(x)\partial_x\right]\widetilde{\Psi}+\left(u_1(t)\cdot\mu'(x)\right)^2\psi_1,& (0,T)\times(0,1), \\
			\widetilde{\xi}(t,0) = \widetilde{\xi}(t,1)=0,& (0,T),\\
			\widetilde{\xi}(0,x)=0,&(0,1).
		\end{array}\right. \end{equation*}
	In \cite[Lemma 4.6, Propositions 4.7 and 4.8]{bournissou2021quad}, Bournissou proved estimates for the auxiliary system when $r=1$. These results can be adapted, without any additional difficulty to our case.
	\begin{defi}\label{defgrandO}
		Given two scalar quantities $A(T,u)$ and $B(T,u)$, we use the notation $A(T,u) =\mathcal{O}(B(T,u))$ as $\left\|u_1\right\|_{L^{\infty}(0,T)}\to 0$ if there exists $ C,\eta,T^*>0$ \textit{s.t.}\ for all $T \in(0,T^*)$, for all
		$u\in L^2((0,T),\rr^r)$ with $\left\|u_1\right\|_{L^{\infty}(0,T)}\leq\eta$, one has $|A(T,u)|\leq C|B(T,u)|$.\end{defi}
	\begin{prop}\label{eqrefformfaibleprimcubrest} For every $p\in\nn^*$ and $\mu\in H^3((0,1),\rr^r)$,  the following error estimates hold, as $\left\|u_1\right\|_{L^{\infty}(0,T)}\to 0$,
		\begin{gather}\label{est0restsharp}
			\left\|\widetilde{\psi}(\cdot;u,\ph_1)-\psi_1\right\|_{L^{\infty}\left((0,T),L^2(0,1)\right)}=\mathcal{O}\left(\left\|u_1\right\|_{L^2(0,T)}\right),	\\
			\left\langle  e^{i u_1(T)\cdot\mu}\left(\widetilde{\psi}(\cdot;u,\ph_1) - \psi_1 -\widetilde{\Psi}\right)(T), \varphi_p\right\rangle=\mathcal{O} \left(\left\|u_1\right\|^2_{L^2(0,T)}\right),\label{est1restsharp}\\
			\left\langle e^{i u_1(T)\cdot\mu}\left(\widetilde{\psi}(\cdot;u,\ph_1) - \psi_1 -\widetilde{\Psi} - \widetilde{\xi}\right)(T), \varphi_p\right\rangle=\mathcal{O}\left(\left\|u_1\right\|^3_{L^2(0,T)}\right).	\label{est2restsharp} 
		\end{gather}
	\end{prop}
	
	To obtain these estimates, we evaluate the quantities $\left\|\psi(\cdot;u;\ph_1)\right\|_{L^{\infty}\left((0,T),L^2(0,1)\right)}$, then \eqref{est0restsharp} followed by \eqref{est1restsharp}, and finally \eqref{est2restsharp}, proceeding step by step. To do this, we rely on the integral formulation of the PDEs and classical inequalities, as well as the isometric properties of $e^{iAs}:L^2(0,1)\to L^2(0,1)$, for every $s\in\rr$.
	
	From the estimates \eqref{est1restsharp} and \eqref{est2restsharp} obtained for the auxiliary system, we derive the following error bounds for \eqref{schr}. This result can be adapted from \cite[Proposition $4.9$]{bournissou2021quad} in the multi-input case. To establish the connection between the auxiliary system and  \eqref{schr}, we use \eqref{lienpsipsitilde}, \eqref{linaux} and \eqref{defauxqua}.
	
	\begin{prop} For every $p \in \nn^*$ and $\mu\in H^3((0,1),\rr^r)$, the following error estimates hold, as $\left\|u_1\right\|_{L^{\infty}(0,T)}\to 0$,
		\begin{gather}	\label{remainderquadu0}	
			\left\|{\psi}(\cdot;u,\ph_1)-\psi_1\right\|_{L^{\infty}\left((0,T),L^2(0,1)\right)}=\mathcal{O}\left(\left\|u_1\right\|_{L^2(0,T)}+|u_1(T)|\right),	\\
			\label{remainderquadu1}	\left\langle \left(\psi(\cdot;u,\ph_1) - \psi_1 -\Psi\right)(T),\varphi_p\right\rangle= \mathcal{O}\left(\left\|u_1\right\|^2_{L^2(0,T)}+|u_1(T)|^2\right),\\\label{remaindercubdu1}	
			\left\langle\left(\psi(\cdot;u,\ph_1) - \psi_1 -\Psi - \xi\right)(T), \varphi_p\right\rangle=\mathcal{O}\left(\left\|u_1\right\|^3_{L^2(0,T)}+|u_1(T)|^3\right). 
		\end{gather}
	\end{prop}
	The final step is to show that the boundary terms  $(u^1_1(T),\cdots,u_1^r(T))$  can be neglected, as they are part of the dynamics. Unlike the previous steps, this one does not follow directly from Bournissou's work \cite[Proposition 4.19]{bournissou2021quad}, which was done for $r=1$ ; therefore, we will explain the details. We recall that the sequences $c^{\ell,L}$ are introduced in Definition \ref{defcj} and the quantities $\omega_j$ and $\nu_j$ are defined in \eqref{omeganu}. 
	\begin{prop}\label{proprestraff} For every $\mu\in H^3\left((0,1),\rr^r\right)$ \textit{s.t.}\ $\left(\mu_{\ell}\right)_{\ell\in\llbracket 1,r\rrbracket}$ is a linearly independent family, then the following error estimate holds, as $\left\|u_1\right\|_{L^{\infty}(0,T)}\to 0$,
		\begin{equation}\label{termrestefin}
			\left|u_1(T)\right|=\mathcal{O}\left(\sqrt{T}\left\|u_1\right\|_{L^2(0,T)}+\left\|\psi(T;u,\ph_1)-\psi_1(T)\right\|_{L^2(0,1)}\right).
		\end{equation}
	\end{prop}
	\begin{proof}
		By hypothesis, there exist $j_1,\cdots,j_{r}\in\nn^*$, pairwise distinct, \textit{s.t.}\ the matrix
		$M:=\left(\langle\mu_{\ell}\ph_1,\ph_{j_n}\rangle\right)_{(\ell,n)\in\llbracket 1,r\rrbracket\times\llbracket 1,r\rrbracket}\in GL_{r}(\rr).$
		Then, using the remainder estimate \eqref{remainderquadu1}, one gets, for all $j\geq 1$, as $\left\|u_1\right\|_{L^{\infty}(0,T)}\to0$,
		\begin{equation}\label{etapetermebord1}\left\langle\psi(T;u,\ph_1)-\psi_1(T),\ph_je^{-i\lambda_1T}\right\rangle=\left\langle\Psi(T),\ph_je^{-i\lambda_1T}\right\rangle+\mathcal{O}\left(\left\|u_1\right\|^2_{L^2(0,T)}+|u_1(T)|^2\right).\end{equation}
		Using the expansion	\eqref{schro:lin:exp}, an integration by parts and the Cauchy--Schwarz's inequality, we obtain for all $j\geq 1$,\begin{equation}\label{etapetermebord2}\left\langle\Psi(T),\ph_je^{-i\lambda_1T}\right\rangle=i\sum_{\ell\in\llbracket 1,r\rrbracket}\langle\mu_{\ell}\ph_1,\ph_j\rangle u_1^{\ell}(T)+\mathcal{O}\left(\sqrt{T}\left\|u_1\right\|_{L^2(0,T)}\right).\end{equation}Focusing on $j\in\{j_1,\cdots,j_r\}$ and using \eqref{etapetermebord2} in \eqref{etapetermebord1} we obtain, as $\left\|u_1\right\|_{L^{\infty}(0,T)}\to0$,
		\begin{equation*}\begin{gathered}M(u_1^{\ell}(T))_{\ell\in\llbracket 1,r\rrbracket}=\mathcal{O}\left(\left\|u_1\right\|_{L^2(0,T)}^2+|u_1(T)|^2+\sqrt{T}\left\|u_1\right\|_{L^2(0,T)}\right.\\\left.+\left\|\psi(T;u,\ph_1)-\psi_1(T)\right\|_{L^2(0,1)}\right).\end{gathered}\end{equation*}
		As $M$ is invertible and $\left\|u_1\right\|_{L^{\infty}(0,T)}\to 0$, Definition \ref{defgrandO} leads to the desired result.
	\end{proof}
	Using \eqref{termrestefin} in \eqref{remainderquadu1} and \eqref{remaindercubdu1}, we finally obtain the following error estimates for the linear and quadratic expansions of the solution.
	\begin{crl}\label{correstcubi} For every  $p \in \nn^*$  and $\mu\in H^3\left((0,1),\rr^r\right)$ \textit{s.t.}\ $\left(\mu_{\ell}\right)_{\ell\in\llbracket 1,r\rrbracket}$ is a linearly independent family, one has, as $\left\|u_1\right\|_{L^{\infty}(0,T)}\to 0$,
		\begin{equation}\label{reste}\begin{gathered}
				\langle\left( \psi(\cdot;u,\ph_1)-\psi_1-\Psi\right)(T),\ph_p\rangle= \mathcal{O}\left(\left\|u_1\right\|_{L^2(0,T)}^2+\left\|\psi(T;u,\ph_1)-\psi_1(T)\right\|_{L^2(0,1)}^2\right),\\
				\langle\left(\psi(\cdot;u,\ph_1)-\psi_1-\Psi-\xi\right)(T),\ph_p\rangle= \mathcal{O}\left(\left\|u_1\right\|_{L^2(0,T)}^3\hspace{-0.06 cm}+\hspace{-0.058 cm}\left\|\psi(T;u,\ph_1)-\psi_1(T)\right\|_{L^2(0,1)}^3\right).\end{gathered}
		\end{equation}
	\end{crl}
	\subsection{A new expression for the quadratic expansion}
	
	The purpose of this section is to show the following proposition, which has already been proved in \cite[Proposition 5.1]{bournissou2021quad} for $\ell=L$ with some minor adaptations required here.
	\begin{prop}\label{devgt} Let $T>0$, $1\leq\ell\leq L\leq r$ and $f,g\in L^2((0,T),\rr)$. If $\mathbf{(H)_{conv}}$ and  $\mathbf{(H)_{null}}$ hold, then, \begin{equation}\begin{gathered}\label{gt}	\Im\left(\mathcal{F}_T^{\ell,L}(f,g)+\mathcal{F}_T^{L,\ell}(g,f)\right)=(-1)^{k+1}\gamma^{\ell,L}_{2k-1}\int_0^Tf_k(t)g_k(t)\cos(\omega_K(t-T))\dt\\+\mathcal{O}\left(\sum_{p=1}^k\left(\left|f_p(T)\right|^2+\left|g_p(T)\right|^2\right)+T\left\|(f_k,g_k)\right\|_{L^2(0,T)}^2\right).\end{gathered}\end{equation}
	\end{prop}
	We first prove the following lemma.
	\begin{lm}\label{lmhg1}
		Let $n\in\nn$ and $H\in\mathcal{C}^{2n}(\rr^2,\cc)$. There exists a quadratic form $S_n$ on $\cc^{3n}$ \textit{s.t.}\ for all $T>0$ and $f,g\in L^1(0,T)$,
		\begin{equation*}\begin{gathered}\int_0^Tf(t)\left(\int_0^t g(\tau)H(t,\tau)\mathrm{d}\tau\right)\dt=-\sum_{p=1}^n\int_0^Tf_p(t)g_{p-1}(t)\partial_1^{p-1}\partial_2^{p-1}H(t,t)\dt\\-\sum_{p=1}^n\int_0^Tf_p(t)g_p(t)\partial_1^p\partial_2^{p-1}H(t,t)\dt+\int_0^Tf_n(t)\left(\int_0^tg_n(\tau)\partial_1^n\partial_2^nH(t,\tau)\mathrm{d}\tau\right)\dt\\+S_n\left(f_1(T),\cdots,f_n(T),g_1(T),\cdots,g_n(T),C_0^n(g),\cdots,C_{n-1}^n(g)\right),\end{gathered}\end{equation*}
		where 	$C^n_p(g):=\displaystyle\int_0^Tg_n(\tau)\partial_1^p\partial_2^nH(T,\tau)\mathrm{d}\tau.$
	\end{lm}
	\begin{proof}
		We prove this lemma by induction on $n\in\nn$. It holds for $n=0$ with $S_0=0$. Assume that the result is true for $n\in\nn$, fixed. Let $H\in\mathcal{C}^{2(n+1)}(\rr^2,\cc)$, $f,g\in L^1(0,T)$. With an integration by parts,
		\begin{equation*}\begin{gathered}\int_0^Tf_n(t)\left(\int_0^tg_n(\tau)\partial_1^n\partial_2^nH(t,\tau)\mathrm{d}\tau\right)\dt=f_{n+1}(T)\int_0^Tg_n(\tau)\partial_1^n\partial_2^nH(T,\tau)\mathrm{d}\tau\\-\int_0^Tf_{n+1}(t)\left(g_n(t)\partial_1^n\partial_2^nH(t,t)+\int_0^tg_n(\tau)\partial_1^{n+1}\partial_2^nH(t,\tau)\mathrm{d}\tau\right)\dt.\end{gathered}\end{equation*}
		Another integration by parts finally leads to
		\begin{equation*}\begin{gathered}\int_0^Tf_n(t)\left(\int_0^tg_n(\tau)\partial_1^n\partial_2^nH(t,\tau)\mathrm{d}\tau\right)\dt=f_{n+1}(T)C_n^n(g)\\-\int_0^Tf_{n+1}(t)g_n(t)\partial_1^n\partial_2^nH(t,t)\dt-\int_0^Tf_{n+1}(t)g_{n+1}(t)\partial_1^{n+1}\partial_2^nH(t,t)\dt\\+\int_0^Tf_{n+1}(t)\left(\int_0^tg_{n+1}(\tau)\partial_1^{n+1}\partial_2^{n+1}H(t,\tau)\mathrm{d}\tau\right)\dt.\end{gathered}\end{equation*}
		An integration by parts gives $C_p^n(g)=g_{n+1}(T)\partial_1^p\partial_2^nH(T,T)-C_p^{n+1}(g)$,  for all $p\in\llbracket 0,n\rrbracket$. We define
		\begin{equation*}\begin{gathered}S_{n+1}:=S_n\left(f_1(T),\cdots,g_n(T),\left(g_{n+1}(T)\partial_1^p\partial_2^nH(T,T)-C_p^{n+1}(g)\right)_{p\in\llbracket 0,n-1\rrbracket}\right)\\+f_{n+1}(T)\left(g_{n+1}(T)\partial_1^n\partial_2^nH(T,T)-C_n^{n+1}(g)\right).\end{gathered}\end{equation*}
		Using the induction hypothesis, we obtain the result.
	\end{proof}
	Thanks to this lemma, we are now able to prove Proposition \ref{devgt}.
	\begin{proof}[Proof of Proposition \ref{devgt}] Using $\mathbf{(H)_{conv}}$, one has $H_{\ell,L},H_{L,\ell}\in\mathcal{C}^{2k}(\rr^2,\cc)$ for all $1\leq \ell, L\leq r$.
		Note that, for all $1\leq \ell\leq L\leq r$, for all $p\in\llbracket 0,2k-1\rrbracket$, for all $t\in[0,T]$, 
		\begin{equation}\label{calcsum1}\partial_1^{\lfloor \frac p2\rfloor}\partial_2^{\lfloor \frac{p+1}{2}\rfloor}H_{L,\ell}(t,t)-\partial_1^{\lfloor \frac{p+1}{2} \rfloor}\partial_2^{\lfloor \frac{p}{2}\rfloor}H_{\ell,L}(t,t)= i^pe^{i\omega_K(t-T)}\gamma^{\ell,L}_p.	\end{equation}
		Let $f,g\in L^2((0,T),\rr)$. Applying Lemma \ref{lmhg1} with $ H_{\ell,L}\to H$, $k \to n$, one has
		\begin{equation}\label{prop2.6.1}\begin{gathered}\int_0^Tf(t)\left(\int_0^t g(\tau)H_{\ell,L}(t,\tau)\mathrm{d}\tau\right)\dt=-\sum_{p=1}^k\int_0^Tf_p(t)g_{p-1}(t)\partial_1^{p-1}\partial_2^{p-1}H_{\ell,L}(t,t)\dt\\-\sum_{p=1}^k\int_0^Tf_p(t)g_p(t)\partial_1^p\partial_2^{p-1}H_{\ell,L}(t,t)\dt+\int_0^Tf_k(t)\left(\int_0^tg_k(\tau)\partial_1^k\partial_2^kH_{\ell,L}(t,\tau)\mathrm{d}\tau\right)\dt\\+S_k\left(f_1(T),\cdots,f_k(T),g_1(T),\cdots,g_k(T),C_0^k(g),\cdots,C_{k-1}^k(g)\right).\end{gathered}\end{equation}
		One recalls that $C_p^n(g)$ is defined in Lemma \ref{lmhg1}.
		Let $p\in\llbracket 1,k\rrbracket$. With an integration by parts,
		\begin{equation*}\begin{gathered}\int_0^Tf_{p-1}(t)g_p(t)\partial_1^{p-1}\partial_2^{p-1}H_{L,\ell}(t,t)\dt=f_p(T)g_p(T)\partial_1^{p-1}\partial_2^{p-1}H_{L,\ell}(T,T)\\-\int_0^Tf_p(t)\left(g_{p-1}(t)\partial_1^{p-1}\partial_2^{p-1}+g_p(t)\left(\partial_1^p\partial_2^{p-1}+\partial_1^{p-1}\partial_2^p\right)\right)H_{L,\ell}(t,t)\dt.\end{gathered}\end{equation*}
		Applying Lemma \ref{lmhg1} with $f\to g$, $g\to f$, $ H_{L,\ell}\to H$, $k\to n$ and using this equality, one finally gets
		\begin{equation}\begin{gathered}\label{lmhg2}\int_0^Tg(t)\left(\int_0^tf(\tau)H_{L,\ell}(t,\tau)\mathrm{d}\tau\right)\dt=\sum_{p=1}^k\int_0^Tf_p(t)g_{p-1}(t)\partial_1^{p-1}\partial_2^{p-1}H_{L,\ell}(t,t)\dt\\+\sum_{p=1}^k\int_0^Tf_p(t)g_p(t)\partial_1^{p-1}\partial_2^pH_{L,\ell}(t,t)\dt+\int_0^Tg_k(t)\left(\int_0^tf_k(\tau)\partial_1^k\partial_2^kH_{L,\ell}(t,\tau)\mathrm{d}\tau\right)\dt\\+\tilde{S_k}\left(f_1(T),\cdots,f_k(T),g_1(T),\cdots,g_k(T),C_0^k(f),\cdots,C_{k-1}^k(f)\right).\end{gathered}\end{equation}
		Using \eqref{calcsum1}, \eqref{prop2.6.1}, \eqref{lmhg2},	 the boundness of $H_{L,\ell}$, the Cauchy--Schwarz's inequality and the assumption $\mathbf{(H)_{null}}$, one gets
		\begin{equation*}\begin{gathered}(-1)^{k+1}e^{i\omega_KT}\left(\mathcal{F}_T^{\ell,L}(f,g)+\mathcal{F}_T^{L,\ell}(g,f)\right)=i\gamma^{\ell,L}_{2k-1}\int_0^Tf_k(t)g_k(t)e^{i\omega_Kt}\dt+\\\sum_{j=1}^{+\infty}\omega_j^k\nu_j^k\int_0^Te^{i\nu_jt}\left(c_j^{\ell,L}f_k(t)\left(\int_0^tg_k(\tau)e^{i\omega_j\tau}\mathrm{d}\tau\right)+c_j^{L,\ell}g_k(t)\left(\int_0^tf_k(\tau)e^{i\omega_j\tau}\mathrm{d}\tau\right)\right)\dt\\+\mathcal{O}\left(\sum_{p=1}^k\left(\left|f_p(T)\right|^2+\left|g_p(T)\right|^2\right)+T\left\|(f_k,g_k)\right\|_{L^2(0,T)}^2\right).\end{gathered}\end{equation*}
		Using the Cauchy--Schwarz's inequality and taking the imaginary part, we obtain the desired result. \end{proof}
	\subsection{Closed-loop estimates}
	The equations \eqref{schro:quad} and \eqref{gt} give a quadratic expansion of the solution to \eqref{schr}. We then estimate the boundary terms $|u_p(T)|^2$ that appear. To do that, we need the following lemma, which is an equivalent in the infinite-dimensional case of \cite[Lemma $3.4$]{gherdaoui2}.
	
	\begin{lm}\label{liberte} Assume that $\mathbf{(H)_{conv}}$, $\mathbf{(H)_{null}}$ and $\mathbf{(H)_{pos}}$ hold. Then, the family 	$$\left(\left(\langle\mu_{\ell}\ph_1,\ph_j\rangle(-i\omega_j)^p\right)_{j\geq 1}\right)_{(p,\ell)\in\llbracket 0,k-1\rrbracket\times\llbracket 1,r\rrbracket}$$ is $\rr$-linearly independent in $\cc^{\nn^*}$.

	\end{lm}
	\begin{proof} By contradiction, assume that there exist $\left(\alpha_{p,\ell}\right)_{(p,\ell)\in\llbracket 0,k-1\rrbracket\times\llbracket 1,r\rrbracket}\in\rr^{k\times r}\setminus\{0\}$ \textit{s.t.}\
		\begin{equation}\label{hypliee}\forall j\geq 1, \quad \sum_{p\in\llbracket 0,k-1\rrbracket}\sum_{\ell\in\llbracket 1,r\rrbracket}\alpha_{p,\ell}\langle\mu_{\ell}\ph_1,\ph_j\rangle(-i\omega_j)^p=0.\end{equation}
		Let $l_0:=\max\{p\in\llbracket 0,k-1\rrbracket; \ (\alpha_{p,1},\cdots,\alpha_{p,r})\neq 0\}.$ By hypothesis \eqref{hypliee}, the following quantity is zero
		\begin{equation*}\sum_{\substack{\ell,L\in\llbracket 1,r\rrbracket,\\p,q\in\llbracket 0,l_0\rrbracket}}\alpha_{p,\ell}\alpha_{q,L}(-i)^{p+q}\sum_{j=1}^{+\infty}\left(c_j^{L,\ell}\omega_j^{p+k-l_0}\nu_j^{q+k-l_0-1}-c_j^{\ell,L}\omega_j^{q+k-l_0-1}\nu_j^{p+k-l_0}\right)=0.\end{equation*}
		We will denote this quantity by $\sum_{\ell,L\in\llbracket 1,r\rrbracket}C_{\ell,L}$. Let $v:=(v_j)_{j\geq 1}$ be a sequence of real numbers and $m,n\in\nn$. Subject to convergence, we define
		$$R_1^{m,n}(v):=\sum_{j=1}^{+\infty}v_j\omega_j^{m+k-l_0}\nu_j^{n+k-l_0-1} \quad\text{and}\quad	R_2^{m,n}(v):=\sum_{j=1}^{+\infty}v_j\omega_j^{m+k-l_0-1}\nu_j^{n+k-l_0}.$$
		By definition of $R_1$ and $R_2$, one has, for all $1\leq\ell\leq L\leq r$,
		\begin{equation*}\begin{gathered}C_{\ell,L}=\sum_{p=0}^{l_0}\alpha_{p,\ell}\alpha_{p,L}(-1)^{p+1}\gamma^{\ell,L}_{2(p+k-l_0)-1}\\+\sum_{\substack{p,q\in\llbracket 0,l_0\rrbracket, \\ p\neq q}}\alpha_{p,\ell}\alpha_{q,L}(-i)^{p+q}\left(R_1^{p,q}(c^{L,\ell})-R_2^{q,p}(c^{\ell,L})\right).\end{gathered}\end{equation*}
		For all $p\in\llbracket 0,l_0-1\rrbracket$, for all $q\in\llbracket p+1,l_0\rrbracket$, Corollary \ref{expan-somm-11} with $q-p-1\to \nu$ and $p+k-l_0\to p$ gives
		$$R_1^{p,q}(c^{L,\ell})-R_2^{q,p}(c^{\ell,L})\in\spn\left(\gamma^{\ell,L}_{2(p+k-l_0)+r}; \ r\in\llbracket 0,q-p-1\rrbracket\right)\subset\rr.$$
		Moreover, $2(p+k-l_0)+r\leq 2k-2$. Using $\mathbf{(H)_{null}}$, this sum is zero. 
		Finally, for all $p\in\llbracket 1,l_0\rrbracket$, for all $q\in\llbracket 0,p-1\rrbracket$, Corollary \ref{expan-somm-11} with $p-q+1\to\nu$ and $q+k-l_0-1\to p$ gives
		$$R_1^{p,q}(c^{L,\ell})-R_2^{q,p}(c^{\ell,L})\in\spn\left(\gamma^{\ell,L}_{2(q+k-l_0-1)+r}; \ r\in\llbracket 0,p-q+1\rrbracket\right)\subset\rr.$$
		Similarly, $2(q+k-l_0-1)+r\leq 2k-2$ so this sum is zero by $\mathbf{(H)_{null}}$. Finally, 
		\begin{equation}\label{exp-crochet4}
			C_{\ell,L}=\alpha_{l_0,\ell}\alpha_{l_0,L}(-1)^{l_0+1}\gamma^{\ell,L}_{2k-1}.
		\end{equation}
		With similar computations, we prove that for all $1\leq L\leq \ell\leq r$,
		\begin{equation}\label{exp-crochet5}
			C_{\ell,L}=\alpha_{l_0,\ell}\alpha_{l_0 ,L}(-1)^{l_0+1}\gamma^{L,\ell}_{2k-1}.
		\end{equation}
		As $\sum_{\ell,L\in\llbracket 1,r\rrbracket}C_{\ell,L}=0$, \eqref{exp-crochet4} and \eqref{exp-crochet5} lead to 
		\begin{equation*}\sum_{\ell\in\llbracket 1,r\rrbracket}\gamma_{2k-1}^{\ell}\frac{\alpha_{l_0,\ell}^2}{2}+\sum_{1\leq \ell<L\leq r}\gamma_{2k-1}^{\ell,L}\alpha_{l_0,\ell}\alpha_{l_0,L}=0,\quad\textit{i.e.}\ \quad q(\alpha_{l_0,1},\cdots,\alpha_{l_0,r})=0.\end{equation*}
		Using $\mathbf{(H)_{pos}}$, $\left(\alpha_{l_0,\ell}\right)_{\ell\in\llbracket 1,r\rrbracket}=0$. This is a contradiction with the choice of $l_0$. We then obtain the result.
	\end{proof}
	\begin{rmq}The proof of Lemma \ref{liberte} may seem a little mysterious at first sight, but there is a strict analogy with the result proved in the finite-dimensional case -- see \cite[Lemma $3.4$]{gherdaoui2}. Note that, if $\mu\in\mathcal{C}^{\infty}_c((0,1),\rr^r)$, then using  \cite[Lemma A.$9$]{gherdaoui}, 
		$$\forall 1\leq\ell\leq r, \ 0\leq p\leq k-1,\quad \left(\langle\mu_{\ell}\ph_1,\ph_j\rangle(-i\omega_j)^p\right)_{j\geq 1}=i^p\left(\left\langle \underline{\ad}_A^p(\mu_{\ell})\ph_1,\ph_j\right\rangle\right)_{j\geq 1}.$$
		We recall that $\underline{\ad}_A^p(B)$ is defined in Section \ref{lienode}. Consequently, we proved that the family $\left(\underline{\ad}_A^p(\mu_{\ell})\ph_1\right)_{(p,\ell)\in\llbracket 0,k-1\rrbracket\times\llbracket 1,r\rrbracket}$ is linearly independent. This is the same family as in the finite-dimensional case. In \cite{gherdaoui2}, we consider $(\alpha_{p,\ell})_{(p,\ell)\in\llbracket 0,k-1\rrbracket\times\llbracket 1,r\rrbracket}\in\rr^{k\times r}$ scalars, not all zero, \textit{s.t.}\
		$$B\ph_1:=\sum_{p\in\llbracket 0,k-1\rrbracket}\sum_{\ell\in\llbracket 1,r\rrbracket}\alpha_{p,\ell}i^p\underline{\ad}_A^p(\mu_{\ell})\ph_1=0.$$
		Then, we define $l_0:=\max\{p\in\llbracket 0,k-1\rrbracket; \ (\alpha_{p,1},\cdots,\alpha_{p,r})\neq 0\}$. As $B\ph_1=0$, one has
		$\left[\underline{\ad}_A^{k-l_0-1}(B),\underline{\ad}_A^{k-l_0}(B)\right]\ph_1=0,$ and we expand this term. Actually, this is exactly what we did in the proof of Lemma \ref{liberte}, in a disguised way, by not assuming regularity on $\mu_{\ell}$ since
		$$\sum_{\ell,L\in\llbracket 1,r\rrbracket}C_{\ell,L}=-\left[\underline{\ad}_A^{k-l_0-1}(B),\underline{\ad}_A^{k-l_0}(B)\right]\ph_1.$$
	\end{rmq}
	Using Lemma \ref{liberte}, we can estimate the boundary terms, as stated in the following proposition.
	\begin{prop}\label{closed-lopp}Assume that the hypotheses $\mathbf{(H)_{reg}}$, $\mathbf{(H)_{conv}}$,  $\mathbf{(H)_{null}}$ and $\mathbf{(H)_{pos}}$ hold. Then, as $\left\|u_1\right\|_{L^{\infty}(0,T)}\to0$,
		\begin{equation*}\sum_{p\in\llbracket 1,k\rrbracket}\left|u_p(T)\right|=\mathcal{O}\left(\left\|u_1\right\|_{L^2(0,T)}^2+\sqrt{T}\left\|u_k\right\|_{L^2(0,T)}+\left\|{\psi(T;u,\ph_1)-\psi}_1(T)\right\|_{L^2(0,1)}\right).\end{equation*}
	\end{prop}
	\begin{proof} By Lemma \ref{liberte}, there exist $j_1,\cdots,j_{kr}\in\nn^*$, pairwise distinct, \textit{s.t.}\
		$$M:=\left(\langle\mu_{\ell}\ph_1,\ph_{j_n}\rangle(-i\omega_{j_n})^p\right)_{((p,\ell),n)\in\llbracket 0,k-1\rrbracket\times\llbracket 1,r\rrbracket\times\llbracket 1,rk\rrbracket}\in GL_{kr}(\cc).$$ 
		Then, using the sharp remainder estimate \eqref{reste} and the expansion of the linearized system given by \eqref{schro:lin:exp}, one has for all $j\geq 1$, as $\left\|u_1\right\|_{L^{\infty}(0,T)}\to0$,
		\begin{equation}\begin{gathered}\label{2.10etap1}\left\langle\psi(T;u,\ph_1)-\psi_1(T),\ph_je^{-i\lambda_1T}\right\rangle=i\sum_{\ell\in\llbracket 1,r\rrbracket}\langle\mu_{\ell}\ph_1,\ph_j\rangle\int_0^Tu^{\ell}(t)e^{i\omega_j(t-T)}\dt\\+\mathcal{O}\left(\left\|u_1\right\|_{L^2(0,T)}^2+\left\|\psi(T;u,\ph_1)-\psi_1(T)\right\|_{L^2(0,1)}^2\right).\end{gathered}\end{equation}
		Using integrations by parts and Cauchy--Schwarz's inequality, one gets, for all $j\geq 1$, $\ell\in\llbracket 1,r\rrbracket$,
		\begin{equation}\label{2.10etap2}\int_0^Tu^{\ell}(t)e^{i\omega_j(t-T)}\dt=\sum_{p=0}^{k-1}(-i\omega_j)^p u_{p+1}^{\ell}(T)+\mathcal{O}\left(\sqrt{T}\left\|u_k\right\|_{L^2(0,T)}\right).\end{equation}
		Using \eqref{2.10etap2} in \eqref{2.10etap1}, one has, for all $j\geq 1$, as $\left\|u_1\right\|_{L^{\infty}(0,T)}\to 0$,
		\begin{equation*}\begin{gathered}\left\langle\psi(T;u,\ph_1)-\psi_1(T),\ph_je^{-i\lambda_1T}\right\rangle=i\sum_{p\in\llbracket 0,k-1\rrbracket}\sum_{\ell\in\llbracket 1,r\rrbracket}\langle\mu_{\ell}\ph_1,\ph_j\rangle(-i\omega_j)^p u_{p+1}^{\ell}(T)\\+\mathcal{O}\left(\left\|u_1\right\|_{L^2(0,T)}^2+\sqrt{T}\left\|u_k\right\|_{L^2(0,T)}+\left\|\psi(T;u,\ph_1)-\psi_1(T)\right\|_{L^2(0,1)}^2\right).\end{gathered}\end{equation*}
		Focusing on $j\in\{j_1,\cdots,j_{kr}\}$, we finally obtain
		\begin{equation*}M\hspace{-0.05 cm}\left(u_p^{\ell}(T)\right)_{\substack{p\in\llbracket 1,k\rrbracket \\\ell\in\llbracket 1,r\rrbracket}}\hspace{-0.05 cm}=\hspace{-0.05 cm}\mathcal{O}\left(\left\|u_1\right\|_{L^2(0,T)}^2+\sqrt{T}\left\|u_k\right\|_{L^2(0,T)}+\left\|\psi(T;u,\ph_1)-\psi_1(T)\right\|_{L^2(0,1)}\right)\hspace{-0.05 cm}.\end{equation*}
		As $M$ is invertible, one obtains the desired result.
	\end{proof}
	\subsection{Interpolation inequality}
	We recall the Gagliardo--Nirenberg interpolation inequalities used in this article and proved in \cite{zbMATH03318089,ASNSP_1959_3_13_2_115_0}.
	\begin{prop}[Gagliardo--Nirenberg inequalities]
		Let $p,q,r,s\in[1,+\infty]$, $ j<l\in\nn$ and $\alpha\in (0,1)$ be \textit{s.t.}\
		$$\frac{j}{l}\leq \alpha\quad{and }\quad \frac{1}{p}=j+\left(\frac 1r-l\right)\alpha+\frac{1-\alpha}{q}.$$
		There exists $C>0$ \textit{s.t.}\ for every $t>0$ and $\ph\in\mathcal{C}^{\infty}([0,t],\rr)$, 
		\begin{equation}\label{gn}\left\|D^j\ph\right\|_{L^p}\leq C\left\| D^l\ph\right\|_{L^r}^{\alpha}\left\|\ph\right\|_{L^q}^{1-\alpha}+Ct^{\frac{1}{p}-j-\frac 1s}\left\|\ph\right\|_{L^s}.
		\end{equation}
	\end{prop}
	
	We have expanded the solution to \eqref{schr} to the second-order term. The remainder term is given by \eqref{reste} and is estimated as $\mathcal{O}\left(\left\|u_1\right\|_{L^2(0,T)}^3\right).$ The purpose of the following lemma is to compare this error term with the drift size $\left\|u_k\right\|_{L^2(0,T)}^2$.
	\begin{lm}\label{gn1} Assume that $k\geq 2$.
		There exists $C>0$  \textit{s.t.}\ for all $T>0$, $f\in H^{2k-3}((0,T),\rr)$,
		$$\left\|f_1\right\|_{L^2(0,T)}^3\leq C\left(1+T^{-2k+3}\right)\left\|f\right\|_{H^{2k-3}(0,T)}\left\|f_k\right\|_{L^2(0,T)}^2.$$
	\end{lm}
	\begin{proof}
		We apply the Gagliardo-Nirenberg interpolation inequalities -- see \cite{zbMATH03318089,ASNSP_1959_3_13_2_115_0} -- with $j=k-1$, $l=3k-3$, $\alpha=\frac{1}{3}$, $p=q=r=s=2$ et $\ph=f_k$ to obtain
		$$\left\|f_1\right\|_{L^2(0,T)}^3\leq C\left\|f^{(2k-3)}\right\|_{L^2(0,T)}\left\|f_k\right\|^2_{L^2(0,T)}+CT^{-3(k-1)}\left\|f_k\right\|_{L^2(0,T)}^3.$$
		Moreover,
		$$\left\|f_k\right\|_{L^2(0,T)}\leq T^k\left\|f\right\|_{L^2(0,T)}\leq T^k\left\|f\right\|_{H^{2k-3}(0,T)}.$$
		Thus, we obtain the desired result.
	\end{proof}
	\subsection{Proof of the drift}We prove Theorem \ref{maintheorem} as a consequence of the following more precise statement.
	\begin{theorem}\label{maintheorem2} Let $k,K,r\in \nn^*$,
		$\mu$ satisfying $\mathbf{(H)_{reg}}$, $\mathbf{(H)_{lin}}$, $\mathbf{(H)_{conv}}$,  $\mathbf{(H)_{null}}$ and $\mathbf{(H)_{pos}}$. If $k\geq 2$ (\textit{resp.}\ $k=1$), there exist $C_1,C_2,T^*>0$ \textit{s.t.}\ for all $T\in(0,T^*)$, there exists $\eta>0$ \textit{s.t.}\ for all $u\in H^{2k-3}((0,T),\rr^r)$ (\textit{resp.}\ $u\in L^2((0,T),\rr^r)$) with $\left\|u\right\|_{H^{2k-3}(0,T)}\leq \eta$ (\textit{resp.}\ $\left\|u_1\right\|_{L^{\infty}(0,T)}\leq \eta$),
		\begin{equation}\label{nonstlc}\begin{gathered}(-1)^{k+1}\text{sgn}(\gamma_{2k-1}^1)\Im\left\langle\psi(T;u,\ph_1),\ph_K e^{-i\lambda_1T}\right\rangle\\\geq C_1\left\|u_k\right\|_{L^2}^2-C_2\left\|\psi(T;u,\ph_1)-\psi_1(T)\right\|_{L^2}^2.\end{gathered}\end{equation}
	\end{theorem}
	\begin{rmq}Theorem \ref{maintheorem2} shows that there exists $0<R<1$ \textit{s.t.}\ the following targets cannot be reached by the solution to \eqref{schr}
		\begin{equation*}\forall\delta\in(0,R), \quad \psi_f:=\left(\sqrt{1-\delta^2}\ph_1+i(-1)^k\text{sgn}(\gamma_{2k-1}^1)\delta\ph_K\right)e^{-i\lambda_1T}.\end{equation*}
		Indeed, if there exists a control $u$ \textit{s.t.}\ $\psi(T;u,\ph_1)=\psi_f$, \eqref{nonstlc} leads to
		$$-\delta\geq C_1\left\|u_k\right\|_{L^2(0,T)}^2-2C_2(1-\sqrt{1-\delta^2})\geq -2C_2\delta^2.$$
		This is impossible when $\delta\to0^+$. Thus, we obtain Theorem \ref{maintheorem}.
	\end{rmq}
	\begin{proof}[Proof of Theorem \ref{maintheorem2}] 
		Using $\mathbf{(H)_{lin}}$ and the remainder estimate \eqref{reste}, the quadratic expansion of the solution gives, as $\left\|u_1\right\|_{L^{\infty}(0,T)}\to 0$,
		\begin{equation*}\Im\left\langle\psi(T;u,\ph_1),\ph_Ke^{-i\lambda_1T}\right\rangle= \Im\langle\xi(T),\ph_Ke^{-i\lambda_1T}\rangle+\mathcal{O}\left(\left\|u_1\right\|_{L^2}^3+\left\|\psi(T;u,\ph_1)-\psi_1(T)\right\|_{L^2}^3\right).\end{equation*}
		Using \eqref{schro:quad} and Proposition \ref{devgt}, one gets, as $\left\|u_1\right\|_{L^{\infty}(0,T)}\to 0$,
		\begin{equation*}\begin{gathered}\Im\langle\psi(T;u,\ph_1),\ph_Ke^{-i\lambda_1T}\rangle=(-1)^{k+1}\int_0^Tq(u_k(t))\cos(\omega_K(t-T))\dt+\\ \mathcal{O}\left(\sum_{p\in\llbracket 1,k\rrbracket}|u_p(T)|^2+T\left\|u_k\right\|_{L^2(0,T)}^2+\left\|u_1\right\|_{L^2(0,T)}^3+\left\|\psi(T;u,\ph_1)-\psi_1(T)\right\|_{L^2(0,1)}^3\right).\\\end{gathered}\end{equation*}
		Using Proposition \ref{closed-lopp} to estimate the boundary terms, one has, as $\left\|u_1\right\|_{L^{\infty}(0,T)}\to 0$,\begin{equation}\begin{gathered}\label{prooffinalavantgn}\Im\left\langle\psi(T;u,\ph_1),\ph_Ke^{-i\lambda_1T}\right\rangle= (-1)^{k+1}\int_0^Tq(u_k(t))\cos(\omega_K(t-T))\dt\\+\mathcal{O}\left(T\left\|u_k\right\|_{L^2(0,T)}^2+\left\|u_1\right\|_{L^2(0,T)}^3+\left\|\psi(T;u,\ph_1)-\psi_1(T)\right\|_{L^2(0,1)}^2\right).\end{gathered}\end{equation}
		Let $T^*:=\frac{\pi}{3\omega_K}$ if $K\neq 1$, $T^*=+\infty$ else. Let $T\in(0,T^*)$. By $\mathbf{(H)_{pos}}$ and Remark \ref{rmqcoer}, there exists $C_1>0$ \textit{s.t.}\ $\text{sgn}(\gamma_{2k-1}^1)q(a)\geq 4C_1\left|a\right|^2$, for every $a\in\rr^r$. Thus,
		\begin{equation}\label{coer} \text{sgn}(\gamma_{2k-1}^1)\int_0^Tq(u_k(t))\cos(\omega_K(t-T))\dt\geq 2C_1\left\|u_k\right\|_{L^2(0,T)}^2. \end{equation}
		\textit{First obstruction --} $k=1$: using \eqref{prooffinalavantgn}, we immediately obtain
		\begin{equation*}\begin{gathered}\Im\left\langle\psi(T;u,\ph_1),\ph_Ke^{-i\lambda_1T}\right\rangle= \int_0^Tq(u_1(t))\cos(\omega_K(t-T))\dt\\+\mathcal{O}\left(\left(T+\left\|u_1\right\|_{L^{\infty}(0,T)}\right)\left\|u_1\right\|_{L^2(0,T)}^2+\left\|\psi(T;u,\ph_1)-\psi_1(T)\right\|_{L^2(0,1)}^2\right).\end{gathered}\end{equation*}
		Using Definition \ref{defgrandO}, there exist $\eta_1,C_2>0$ \textit{s.t.}\ for all $T\in(0,1)$, for all $u\in L^2((0,T),\rr^r)$ satisfying $\left\|u_1\right\|_{L^{\infty}(0,T)}\leq\eta_1$,
		\begin{equation*}\begin{gathered}\left|\Im\left\langle\psi(T;u,\ph_1),\ph_Ke^{-i\lambda_1T}\right\rangle-\int_0^Tq(u_1(t))\cos(\omega_K(t-T))\dt\right|\\\leq C_2\left(T+\left\|u_1\right\|_{L^{\infty}(0,T)}\right)\left\|u_1\right\|_{L^2(0,T)}^2+C_2\left\|\psi(T;u,\ph_1)-\psi_1(T)\right\|_{L^2(0,1)}^2.\end{gathered}\end{equation*}
		We fix $T_f:=\min\left(1,\frac{C_1}{2C_2}\right)$ and $\eta:=\min\left(\eta_1,\frac{C_1}{2C_2}\right)$. This equation together with \eqref{coer} conclude the proof.
		\medskip\\
		\noindent
		\textit{Other obstructions --} $k\geq 2$:
		using the interpolation inequality proved in Lemma \ref{gn1} in \eqref{prooffinalavantgn}, one gets, as $\left\|u_1\right\|_{L^{\infty}(0,T)}\to 0$,
		\begin{equation*}\begin{gathered}(-1)^{k+1}\Im\left\langle\psi(T;u,\ph_1),\ph_Ke^{-i\lambda_1T}\right\rangle= \int_0^Tq(u_k(t))\cos(\omega_K(t-T))\dt\\+\mathcal{O}\left(\left(T+(1+T^{-2k+3})\left\|u\right\|_{H^{2k-3}(0,T)}\right)\left\|u_k\right\|_{L^2(0,T)}^2+\left\|\psi(T;u,\ph_1)-\psi_1(T)\right\|_{L^2(0,1)}^2\right).\end{gathered}\end{equation*}
		Then, there exist $\eta_1,C_2>0$ \textit{s.t.}\ for all $T\in(0,T^*)$, for all $u\in H^{2k-3}((0,T),\rr^r)$ satisfying $\left\|u_1\right\|_{L^{\infty}(0,T)}\leq\eta_1$,
		\begin{equation}\begin{gathered}\label{fp1}\left|(-1)^{k+1}\Im\left\langle\psi(T;u,\ph_1),\ph_Ke^{-i\lambda_1T}\right\rangle-\int_0^Tq(u_k(t))\cos(\omega_K(t-T))\dt\right|\leq\\ C_2\left(T+(1+T^{-2k+3})\left\|u\right\|_{H^{2k-3}(0,T)}\right)\left\|u_k\right\|_{L^2(0,T)}^2+C_2\left\|\psi(T;u,\ph_1)-\psi_1(T)\right\|_{L^2(0,1)}^2.\end{gathered}\end{equation}
		Let  $T_f:=\min\left(T^*,\frac{C_1}{3C_2}\right)$ and $T\in(0,T_f)$. We define $\eta:=\min\left(\eta_1,\frac{C_1}{3C_2},\frac{C_1}{3C_2}T^{2k-3}\right)$. Then, for all $u\in H^{2k-3}((0,T),\rr^r)$ with $\left\|u\right\|_{H^{2k-3}(0,T)}\leq\eta$, \eqref{fp1} leads to
		\begin{equation*}\begin{gathered}\left|(-1)^{k+1}\Im\left\langle\psi(T;u,\ph_1),\ph_Ke^{-i\lambda_1T}\right\rangle- \int_0^Tq(u_k(t))\cos(\omega_K(t-T))\dt\right|\\\leq C_1\left\|u_k\right\|_{L^2(0,T)}^2+C_2\left\|\psi(T;u,\ph_1)-\psi_1(T)\right\|_{L^2(0,1)}^2.\end{gathered}\end{equation*}
		This equation together with \eqref{coer} lead to the desired result.
	\end{proof}
	\section{Comparison with the finite-dimensional case}\label{lienode}
	
	The main result of this article is an adaptation to \eqref{schr} of a theorem that proves quadratic obstructions to STLC for multi-input control-affine systems in \cite{gherdaoui2} for $r=2$. Here is the framework and key definitions. One considers the system
	\begin{equation}\label{affine-syst}x'(t)=f_0(x(t))+\sum_{\ell=1}^ru^{\ell}(t)f_{\ell}(x(t)),
	\end{equation}where the state is $x(t)\in\rr^d$, the controls are scalar functions $u^{\ell}(t)$ and $f_{\ell}$ are real-analytic vector fields on a neighborhood of $0$ in $\rr^d$ \textit{s.t.}\ $f_0(0)=0$. The last hypothesis ensures that $0$ is an equilibrium of the free system -- \textit{i.e.}\ with $u\equiv 0$. 
	
	For each $t>0$, $u\in L^1((0,t),\rr^r)$, there exists a unique mild solution defined on a maximal interval time to \eqref{affine-syst} with initial data $0$ at time $t=0$, which we will denote by $x(\cdot;u)$. As we are interested in small-time and small controls, the solution is well-defined up to time $t$.
	
	\begin{defi}[$W^{m,p}$-STLC]
		\label{stlc} Let $m\in\llbracket -1,+\infty\llbracket$, $p\in[1,+\infty]$. System \eqref{affine-syst} is $W^{m,p}$-STLC (around $0$) if for all $T,\ep>0$, there exists $\delta>0$ \textit{s.t.}\ for all $x_f\in B(0,\delta)$, there exists $u\in W^{m,p}((0,T),\rr^r)\cap L^1((0,T),\rr^r)$ with $\left\|u\right\|_{W^{m,p}}\leq \ep$ \textit{s.t.}\ $x(T;u)=x_f$.
	\end{defi}
	The historical definition of STLC is the case where $m=0$ and $p=\infty$ -- see \cite{coronbook}.
	Let $X:= \{X_0,\cdots,X_r\}$ be a set of $r+1$ non-commutative indeterminates.
	\begin{defi}[Free Lie algebra]
		\label{def:free-algebra}
		We consider $\mathcal{A}(X)$ the unital associative algebra of polynomials of the indeterminates in $X$. For $a,b\in\mathcal{A}(X)$, one defines the Lie bracket of $a$ and $b$ as $[a,b]:= ab - ba$. We denote by $\mathcal{L}(X)$ the free Lie algebra generated by $X$ over the field $\rr$.
	\end{defi}
	\begin{defi}[Lie bracket of vector fields]
		Let $f,g:\Omega\to\rr^d$ be $\mathcal{C}^{\infty}$ vector fields on an open subset $\Omega$ of $\rr^d$. One defines	\begin{equation*}
			[f,g]:x\in\Omega\mapsto \mathrm{D}g(x) \cdot f(x) - \mathrm{D}f(x) \cdot g(x).
		\end{equation*}
	\end{defi}
	\begin{defi}[Evaluated Lie bracket]
		\label{Def:evaluated_Lie_bracket}
		Let $f_0,\cdots,f_r$ be $\mathcal{C}^\infty$ vector fields on an open subset $\Omega$ of $\rr^d$.
		For $b \in \mathcal{L}(X)$, we define $f_b:=\Lambda(b)$, where $\Lambda:\mathcal{L}(X) \to \mathcal{C}^\infty(\Omega,\rr^d)$ is the unique homomorphism of Lie algebras \textit{s.t.}\ $\Lambda(X_{\ell})=f_{\ell}$ for $0\leq\ell\leq r$.
	\end{defi}
	For example, if $b=[[[X_1,X_r],X_0],[X_{r-1},X_0]]$, one has $f_b=[[[f_1,f_r],f_0],[f_{r-1},f_0]]$.
	\begin{defi}[Bracket integration $b0^\nu$] \label{def:0nu}
		For $b \in\mathcal{L}(X)$ and $\nu \in \nn$, we use the short-hand $b 0^\nu$ to denote the right-iterated bracket $[\cdots[b, X_0], \cdots, X_0]$, where $X_0$ appears $\nu$ times.
	\end{defi}
	For example, if $b=[[X_{r},X_0],[X_0,X_r]]$, then $b0^2=[[[[X_r,X_0],[X_0,X_r]],X_0],X_0]$.
	\begin{defi} We define the following brackets
		\begin{enumerate}
			\item[] $M_l^{\ell}:=X_{\ell}0^l$; \ $(\ell,l)\in\llbracket1,r\rrbracket\times\nn$,
			\item[] $C_{p,l}^{\ell,L}:=(-1)^p\left[M_{\lfloor\frac{p+1}{2}\rfloor}^{\ell},M_{\lfloor\frac{p}{2}\rfloor}^L\right]0^l$; \ $(\ell,L)\in\llbracket 1,r\rrbracket^2$ \textit{s.t.}\ $\ell\leq L$, $(p,l)\in\nn^2$.
		\end{enumerate}
	\end{defi}
	The following theorem generalizes \cite[Theorem $1.22$]{gherdaoui2}, originally stated for $r=2$. By applying the same method as the one introduced in \cite{beauchard2024unified}, we obtain the extended result presented here, which relies on a Magnus-type representation formula of the state.\footnote{A proof of this theorem is available in my thesis manuscript.} The proof is designed to prepare an easier transfer to PDEs.
	\begin{theorem}\label{dimfinie} Let $k,r\in\nn^*$ and $f_0,\cdots,f_r$ be real-analytic vector fields on a neighborhood of $0$ in $\rr^d$ \textit{s.t.}\ $f_0(0)=0$.
		We define the vector subspace of $\rr^d$ \begin{equation*}\begin{gathered}{N}_k:=\spn\left\lbrace f_{M_l^{\ell}}(0); \hspace{0.1 cm} (\ell,l)\in\llbracket 1,r\rrbracket\times\nn \right\rbrace\\+\spn\left\lbrace f_{C^{\ell,L}_{p,l}}(0); \hspace{0.1 cm} (\ell,L)\in\llbracket 1,r\rrbracket^2\text{ \textit{s.t.}\ }\ell\leq L, \ (p,l)\in\llbracket 0,2k-2\rrbracket\times\nn\right\rbrace.\end{gathered}\end{equation*}
		Assume that there exists a linear form $\mathbb{P}$ on $\rr^d$ \textit{s.t.}\ $\restriction{\mathbb{P}}{{N}_k}\equiv0$ and the following quadratic form is definite on $\rr^r$ $$(a_1,\cdots,a_r)\in\rr^r\mapsto \sum_{\ell\in\llbracket 1,r\rrbracket}\mathbb{P}\left(f_{C_{2k-1,0}^{\ell,\ell}}(0)\right)\frac{a_{\ell}^2}{2}+\sum_{1\leq\ell<L\leq r} \mathbb{P}\left(f_{C_{2k-1,0}^{\ell,L}}(0)\right)a_{\ell}a_L\in\rr.$$
		Then, the affine system \eqref{affine-syst} is not $H^{2k}$-STLC around $0$.
	\end{theorem}
	\begin{rmq}\label{rmqdimfiniavectheo}
		We illustrate Theorem \ref{dimfinie} with the example of system \eqref{jouetex} with $r=2$ and $k=1$. 	One has $[f_1,f_2](0)=0$. Then, for all $l\in\nn$, $f_{C^{1,2}_{0,l}}(0)=0$. Thus, ${N}_k=\spn\left(e_1,e_2,e_3\right).$
		Finally, one has $f_{C^{1,1}_{1,0}}(0)=2e_4$, $f_{C^{2,2}_{1,0}}(0)=4e_4$ and $f_{C_{1,0}^{1,2}}(0)=\frac{1}{2}e_4$. Then, the linear form $\mathbb{P}:=\langle \cdot,e_4\rangle$ is suitable because the associated quadratic form is positive definitive on $\rr^2$ as, for every $(x,y)\in\rr^2$, 
		$$x^2+\frac{1}{2}xy+2y^2\geq\frac{3}{4}x^2+\frac{7}{4}y^2.$$
	\end{rmq}
	\medskip
	
	We now give an interpretation of the series $\gamma_{p}^{\ell,L}$ in terms of Lie brackets. For $A$ and $B$ two operators, we define $\underline{\ad}^l_A(B)$ as $\underline{\ad}_A^0(B):=B$ and $\underline{\ad}_A^{l+1}(B):=\underline{\ad}^l_A(B)A-A\underline{\ad}^l_A(B)$. Under appropriate assumptions  on the functions $\mu_{\ell}$ for domain compatibility (a finite number of derivatives of odd order have to vanish on the boundary), $\gamma_{p}^{\ell,L}$ can be interpreted as Lie brackets as, for all $0\leq p\leq 2k-1$, $1\leq\ell\leq L\leq r,$
	\begin{equation}\label{sommecrochet} \gamma_p^{\ell,L}=(-1)^p\left\langle[\underline{\ad}_{A}^{\lfloor\frac{p+1}{2}\rfloor}(B_{\ell}),\underline{\ad}_{A}^{\lfloor\frac{p}{2}\rfloor}(B_L)]\ph_1,\ph_K\right\rangle,\end{equation} where $A$ is defined in \eqref{laplacien} and $B_{\ell}$ is the multiplication operator by $\mu_{\ell}$ in $L^2(0,1)$.
	We refer to \cite[Propositions  A.$7$, A.$8$, A.$10$ and A.$11$]{gherdaoui} for a precise proof. Using this interpretation, we can recognize in \eqref{heurisart3termdom} the leading terms of the Magnus-type representation formula that is used to prove Theorem \ref{dimfinie} -- see \cite{Beauchard_2023} for more details.
	
	\medskip
	
	Let us make the link between Theorems \ref{maintheorem} and \ref{dimfinie}. When $\ell<L$, the Lie brackets $\gamma_p^{\ell,L}$ are specific to the multi-input system \eqref{schr} with $r>1$. They can be used to recover a complex lost direction at linear order -- see \cite{gherdaoui} with $r=2$.
	They can be considered as good in a STLC point of view. \textit{A contrario}, the Lie brackets $\gamma_{2k-1}^{\ell,\ell}$ are known to be potential obstructions to STLC for the bilinear Schrödinger equation in the single-input case, in an appropriate functional setting -- see \cite{bournissou2021quad} with $r=1$. Assume that $\mu\in\mathcal{C}^{\infty}_c((0,1),\rr^r)$. Heuristically, we can think of the bilinear Schrödinger equation \eqref{schr} as a control-affine system of the form \eqref{affine-syst} with $f_0=A$ -- where the operator $A$ is defined in \eqref{laplacien} -- and $f_{\ell}=B_{\ell}$ with $1\leq\ell\leq r$, where $B_{\ell}$ is the multiplication operator by $\mu_{\ell}$ on $L^2(0,1)$. The equilibrium is no longer $0$ but the free trajectory $\psi_1=\psi(\cdot;0,\ph_1)$.
	In this situation, the linear form is $\mathbb{P}:=\langle\cdot,\ph_K\rangle$. The hypotheses are the same in the finite and infinite-dimensional cases. 
	\begin{enumerate}
		\item The hypothesis $\restriction{\mathbb{P}}{{N}_k}\equiv 0$ : \eqref{sommecrochet}, \cite[Lemma A.$9$]{gherdaoui} and the assumptions $\mathbf{(H)_{lin}}$ and $\mathbf{(H)_{null}}$ give, for all  $(\ell,l)\in\llbracket 1,r\rrbracket\times\nn$,
		$$\mathbb{P}\left( \underline{\ad}^l_A(B_{\ell})\ph_1\right)=\left\langle \underline{\ad}^l_A(B_{\ell})\ph_1,\ph_K\right\rangle=(\lambda_1-\lambda_K)^l\langle\mu_{\ell}\ph_1,\ph_K\rangle =0,$$
		and for all $(\ell,L)\in\llbracket 1,r\rrbracket^2$ \textit{s.t.}\ $\ell\leq L$, for all $(p,l)\in\llbracket 0,2k-2\rrbracket\times\nn,$  $$\mathbb{P}\left(\underline{\ad}_A^l\left([\underline{\ad}_A^{\lfloor\frac{p+1}{2}\rfloor}(B_{\ell}),\underline{\ad}_A^{{\lfloor\frac{p}{2}\rfloor}}(B_L)]\right)\ph_1\right)=(\lambda_1-\lambda_K)^l(-1)^p\gamma_p^{\ell,L}=0.$$
		\item Equation \eqref{sommecrochet} and $\mathbf{(H)_{pos}}$ ensure that the hypotheses about the quadratic forms are the same.
	\end{enumerate}
	\begin{rmq}
		The functional setting is different. In Theorem \ref{dimfinie}, we deny $H^{2k}$-STLC, whereas in Theorem \ref{maintheorem}, we disprove the $H^{2k-3}$-STLC. These choices of functional settings are determined through interpolation theory by the competition between the negative Sobolev norm quantifying the quadratic drift and the size of the cubic remainder. For \eqref{schr}, we obtain a very good error estimate involving a Lebesgue norm of the cube of the control's primitive -- see Corollary \ref{correstcubi}.
		While classical for scalar-input systems, such an estimate fails in general for multi-input systems, even for ODEs -- see \cite[Section 7.5]{Beauchard_2023}. The key point is that, for \eqref{schr}, the operators $B_\ell$ commute. Taking advantage of this property allows us to obtain an error estimate similar to the scalar-input case by using an auxiliary system.
	\end{rmq}
	\appendix
	\section{Postponed proofs}	\subsection{Existence of $\mu_1,\cdots,\mu_r$ verifying the hypotheses}\label{existe}
	\begin{rmq}\label{rmqmuexplicite} Let $r=2$, $k=1$ and $K\geq 9$ be an odd integer. In this setting, we can give an explicit example of such functions $\mu_1$ and $\mu_2$. Let $n:=\frac{K+3}{2}\in\nn_{\geq 6}$. We define
		$$\mu_1:=\ph_4+\ph_{2(n-4)}\quad\text{and}\quad \mu_2:=\ph_2+\ph_{2(n-1)}.$$
		The hypothesis $\mathbf{(H)_{reg}}$ is then verified since $\mu\in\mathcal{C}^{\infty}((0,1),\rr^r)$. A parity argument gives $\langle\mu_1\ph_1,\ph_K\rangle=\langle\mu_2\ph_1,\ph_K\rangle=0$ and $\mathbf{(H)_{lin}}$ holds. Moreover, $\mathbf{(H)_{null}}$ is empty for $k=1$. Using Remark \ref{rmqassympto}, $\mathbf{(H)_{conv}}$ holds as $\ph_j\ph_l\in H^2_{(0)}(0,1)$ for every $j,l\in\nn^*$. Finally, using \eqref{sommecrochet}, one has 
		$$\gamma_1^{\ell}=-2\left\langle \left(\mu_{\ell}'\right)^2\ph_1,\ph_K\right\rangle, \ \ell\in\{1,2\} \quad\text{and}\quad \gamma_1^{1,2}=-2\langle \mu_1'\mu_2'\ph_1,\ph_K\rangle.$$
		An explicit computation leads to $\gamma_1^{1,2}=0$,  $\gamma_1^1=\left\lbrace\begin{array}{cc}-16(n-4)\pi^2&\text{ if }n\geq 8\\-12\pi^2&\text{ if }n=7\\-64\pi^2&\text{ if }n=6\end{array}\right.$ and $\gamma_1^2=-8(n-1)\pi^2$. Using Remark \ref{rmqliner=2}, $\mathbf{(H)_{pos}}$ holds.
	\end{rmq}
	\begin{theorem}\label{theoremexistence} Let $k,K,r\in\nn^*$. There exist $\mu_1,\cdots,\mu_r$ satisfying the hypotheses $\mathbf{(H)_{reg}}$, $\mathbf{(H)_{lin}}$,  $\mathbf{(H)_{conv}}$,   $\mathbf{(H)_{null}}$ and $\mathbf{(H)_{pos}}$.
	\end{theorem}
	We use arguments very similar to those developed in \cite{bournissou2021quad,bournissou2022smalltime,gherdaoui}. That is why we will only give a proof skeleton for $r=2$.
	\begin{proof}[Idea of the proof] We prove more precisely the existence of $\mu_1,\mu_2\in\mathcal{C}^{\infty}_c((0,1),\rr)$ \textit{s.t.}\ the previous assumptions are satisfied. If  $\mu_1,\mu_2\in\mathcal{C}^{\infty}_c((0,1),\rr)$, $\mathbf{(H)_{reg}}$ is already verified. This is also the case for  $\mathbf{(H)_{conv}}$ by Remark \ref{rmqassympto}.
		In \cite{gherdaoui}, using a trick that divides function supports into two parts, we have explained how to guarantee the existence of functions $\mu_1,\mu_2\in\mathcal{C}^{\infty}_c((0,1),\rr)$ so that the following equalities hold
		\begin{eqnarray}\label{hyp}&\langle\mu_1\ph_1,\ph_K\rangle=\langle\mu_2\ph_1,\ph_K\rangle=0,\\
			\label{hyp2}&\forall 1\leq p\leq 2k-2, \ \forall 1\leq\ell\leq L\leq 2, \quad \gamma_p^{\ell,L}=0,\\
			\label{hyp3}& \gamma^{1,2}_{2k-1}=0.\end{eqnarray}
		As a consequence, $\mathbf{(H)_{lin}}$ and $\mathbf{(H)_{null}}$ are satisfied. Moreover using Remark \ref{rmqliner=2} and \eqref{hyp3}, $\mathbf{(H)_{pos}}$ becomes
		$$\gamma_{2k-1}^1\gamma_{2k-1}^2>0.$$
		Then, we want to guarantee that \eqref{hyp}, \eqref{hyp2}, \eqref{hyp3},  $\gamma_{2k-1}^1>0$ and $\gamma_{2k-1}^2>0$ can be satisfied simultaneously to conclude the proof. This is what Bournissou did in \cite[Theorem A.$4$]{bournissou2021quad}. By adapting her method to our setting, we obtain the claimed result.
	\end{proof}
	\subsection{Some series expansions}
	We recall that $\omega_j$ and $\nu_j$ are defined in \eqref{omeganu}.
	\begin{defi}	Let $a:=(a_j)_{j\geq 1},b:=(b_j)_{j\geq 1}$ be sequences of real numbers and $l\in\nn$  \textit{s.t.}\ $\left(a_jj^{2l}\right)_{j\geq 1}$,	$\left(b_jj^{2l}\right)_{j\geq 1}\in\ell^1(\nn^*)$. We define
		\begin{equation*}\gamma_l(a,b):=\displaystyle\sum_{j=1}^{+\infty}\left(a_j\omega_j^{\lfloor\frac{l}{2}\rfloor}\nu_j^{\lfloor\frac{l+1}{2}\rfloor}-b_j\omega_j^{\lfloor\frac{l+1}{2}\rfloor}\nu_j^{\lfloor\frac{l}{2}\rfloor}\right).\end{equation*}
	\end{defi}
	\begin{rmq} We recall that $\gamma_p^{\ell,L}$ is specified in Definition \ref{defquadbra}. If $\mathbf{(H)_{conv}}$ holds, 
		$$\forall 0\leq p \leq 2k-1,\ 1\leq\ell\leq L\leq r, \qquad \gamma_p^{\ell,L}=\gamma_p\left(c^{\ell,L},c^{L,\ell}\right).$$
	\end{rmq}
	\begin{lm}\label{expan-somm-1}Let $\nu\in\nn$. There exist coefficients $\left(\beta_l^{\nu}\right)_{l\in\llbracket 0,\nu\rrbracket}\in\zz^{\nu+1}$ \textit{s.t.}\ for every $p\in\nn$ and $a:=(a_j)_{j\geq 1},b:=(b_j)_{j\geq 1}\in\rr^{\nn^*}$  verifying $\displaystyle\sum_{j=1}^{+\infty}\left(|a_j|+|b_j|\right)j^{2(2p+\nu)}<+\infty$,
		\begin{equation*}\sum_{j=1}^{+\infty}\left(a_j\omega_j^{p+\nu}\nu_j^p-b_j\omega_j^p\nu_j^{p+\nu}\right)=\sum_{l=0}^{\nu}\beta_l^{\nu}(-1)^l\gamma_{2p+l}\left(a,b\right)\omega_K^{\nu-l}.\end{equation*}
	\end{lm}
	\begin{proof}
		We prove this statement by double induction on $\nu\in\nn$. 
		\textit{Initialisation}. For
		$\nu=0$, the desired equality is true by definition with $\beta_0^0=1$. Let $\nu=1$. As for all $j\in\nn^*$, $\omega_K-\nu_j=\omega_j$, one gets
		\begin{align*}\sum_{j=1}^{+\infty}\left(a_j\omega_j^{p+1}\nu_j^p-b_j\omega_j^p\nu_j^{p+1}\right)&=\sum_{j=1}^{+\infty}a_j\omega_j^p\nu_j^p(\omega_K-\nu_j)-\sum_{j=1}^{+\infty}b_j\omega_j^p\nu_j^p(\omega_K-\omega_j)\\&=\omega_K\gamma_{2p}(a,b)-\gamma_{2p+1}(a,b).\end{align*}
		We obtain the result with $\beta_0^1=\beta_1^1=1$. \textit{Induction step:} assume that the result holds for $\nu$ and $\nu+1$. Let $p\in\nn$, $(a_j)_{j\geq 1},(b_j)_{j\geq 1}\in\rr^{\nn^*}$ be \textit{s.t.}\ $\displaystyle\sum_{j=1}^{+\infty}\left(|a_j|+|b_j|\right)j^{2(2p+\nu+2)}<+\infty$. Using the same strategy,
		\begin{equation*}\begin{gathered}
				\sum_{j=1}^{+\infty}\left(a_j\omega_j^{p+\nu+2}\nu_j^p-b_j\omega_j^p\nu_j^{p+\nu+2}\right)=\omega_K\sum_{j=1}^{+\infty}\left(a_j\omega_j^{p+\nu+1}\nu_j^p-b_j\omega_j^p\nu_j^{p+\nu+1}\right)\\-\sum_{j=1}^{+\infty}\left(\left(a_j\omega_j\nu_j\right)\omega_j^{p+\nu}\nu_j^p-\left(b_j\omega_j\nu_j\right)\omega_j^p\nu_j^{p+\nu}\right).\end{gathered}
		\end{equation*}
		We use the equality $\gamma_{2p+l}\left(\left(a_j\omega_j\nu_j\right)_{j\geq 1},\left(b_j\omega_j\nu_j\right)_{j\geq 1}\right)=\gamma_{2p+l+2}(a,b)$ and the induction hypothesis to obtain the result, with $\beta_l^{\nu+2}=\beta_l^{\nu+1}-\beta_{l-2}^{\nu}$.
	\end{proof}
	\begin{crl}\label{expan-somm-11}Assume that $\mathbf{(H)_{conv}}$ holds. Let $\nu\in\nn$,
		there exist $\left(\beta_l^{\nu}\right)_{l\in\llbracket 0,\nu\rrbracket}$,
		$\left(\delta_l^{\nu}\right)_{l\in\llbracket 0,\nu\rrbracket}\in\zz^{\nu+1}$ \textit{s.t.}\ for all $p\in\nn$ satisfying $2p+\nu\leq 2k-1$, for every $1\leq\ell\leq L\leq r$,
		$$\sum_{j=1}^{+\infty}\left(c_j^{\ell,L}\omega_j^{p+\nu}\nu_j^p-c_j^{L,\ell}\omega_j^p\nu_j^{p+\nu}\right)=\sum_{l=0}^{\nu}\beta_l^{\nu}(-1)^l\gamma^{\ell,L}_{2p+l}\omega_K^{\nu-l},$$
		$$\sum_{j=1}^{+\infty}\left(c_j^{L,\ell}\omega_j^{p+\nu}\nu_j^p-c_j^{\ell,L}\omega_j^p\nu_j^{p+\nu}\right)=\sum_{l=0}^{\nu}\delta_l^{\nu}(-1)^l\gamma_{2p+l}^{\ell,L}\omega_K^{\nu-l}.$$
	\end{crl}
	Moreover, $\beta_0^1=\beta_1^1=1$.
	\begin{proof}
		The first point is a direct consequence of Lemma \ref{expan-somm-1} with $a=c^{\ell,L}$ and $b=c^{L,\ell}$. The second one can be proved in the same way (induction).
	\end{proof}
	\begin{rmq} Assume that $\mu\in\mathcal{C}^{\infty}_c((0,1),\rr^r)$. Then, $\mathbf{(H)_{conv}}$ holds by Remark \ref{rmqassympto}. Furthermore, for all $p,\nu\in\nn$ with $2p+\nu\leq 2k-1$, for all $1\leq\ell\leq L\leq r$, a bracket computation gives
		\begin{equation}\label{dimfini-inf-exp-somm2}\left\langle [\underline{\ad}_{A}^p(\mu_{\ell}),\underline{\ad}_A^{p+\nu}(\mu_L)]\ph_1,\ph_K\right\rangle=(-1)^{\nu}\sum_{j=1}^{+\infty}\left(c_j^{\ell,L}\omega_j^{p+\nu}\nu_j^p-c_j^{L,\ell}\omega_j^p\nu_j^{p+\nu}\right).\end{equation}
		Using \eqref{sommecrochet} and \eqref{dimfini-inf-exp-somm2}, the first expansion of Corollary \ref{expan-somm-11} can be written as
		\begin{equation*}\begin{gathered}\left\langle[\underline{\ad}_{A}^p(\mu_{\ell}),\underline{\ad}_A^{p+\nu}(\mu_L)]\ph_1,\ph_K\right\rangle\\=\sum_{l=0}^{\nu}\beta_l^{\nu}(-1)^l\left\langle\underline{\ad}_A^{\nu-l}\left(\left[\underline{\ad}_A^{p+\lfloor\frac{l+1}{2}\rfloor}(\mu_{\ell}),\underline{\ad}_A^{p+\lfloor\frac{l}{2}\rfloor}(\mu_L)\right]\right)\ph_1,\ph_K\right\rangle.\end{gathered}\end{equation*}
		Thus, Corollary \ref{expan-somm-11} is the equivalent of \cite[Lemma A.$6.3$]{gherdaoui2} in the infinite-dimensional case. Moreover, the sequences defined in Corollary \ref{expan-somm-11} and \cite[Lemma A.$6.3$]{gherdaoui2} are the same.
	\end{rmq}
	\section{Small-state small-time local controllability}
	\begin{defi} We say that system \eqref{schr} is small-state STLC in $H^3_{(0)}(0,1)$ around the ground state if for every $T, \varepsilon>0$, there exists $\delta>0$ \textit{s.t.}\ for every $\psi_f \in H^3_{(0)}(0,1)\cap \mathcal{S}$ with $\|\psi_f - \psi_1(T)\|_{H^3}<\delta$, there exists $u\in L^2((0,T),\rr^r)$ \textit{s.t.}\ $\psi(T;u,\varphi_1)=\psi_f$ and $\left\|\psi(\cdot;u,\ph_1)-\psi_1\right\|_{L^{\infty}\left((0,T),L^2(0,1)\right)}\leq \ep$.
	\end{defi}
	\begin{theorem}\label{theoremsmallstate}  Let $r\in\nn^*$ and $\mu\in H^3\left((0,1),\rr^r\right)$ \textit{s.t.}\ $\left(\mu_{\ell}\right)_{\ell\in\llbracket 1,r\rrbracket}$ is a linearly independent family. System \eqref{schr} is small-state STLC  in $H^3_{(0)}(0,1)$ around the ground state iff it is $W^{-1,\infty}$-STLC in $H^3_{(0)}(0,1)$ around the ground state.
	\end{theorem}
	\begin{rmq} Such an equivalence between these two notions of STLC has already been proved in \cite[Section $8$]{beauchard2017quadratic} for single-input control-affine systems \textit{i.e.}\ \eqref{affine-syst} with $r=1$.
	\end{rmq}
	\begin{proof} By Proposition \ref{proprestraff}, there exist $C,\eta,T^*>0$ \textit{s.t.}\ for all $s\in (0,T^*)$, $u\in L^2\left((0,s),\rr^r\right)$ \textit{s.t.}\ $\left\|u_1\right\|_{L^{\infty}(0,s)}\leq\eta$,
		\begin{equation}\label{small-state1}\left|u_1(s)\right|\leq C\left(\sqrt{s}\left\|u_1\right\|_{L^2(0,s)}+\left\|\psi(s;u,\ph_1)-\psi_1(s)\right\|_{L^2(0,1)}\right).\end{equation}
		Let us prove the direct implication. We show that \eqref{schr} is $W^{-1,\infty}$-STLC in $H^3_{(0)}(0,1)$ around the ground state: let $\ep>0$ and $T>0$ be \textit{s.t.}\ $(CT)^2\leq \frac 12$, $T<T^*$ and $\ep\leq\eta$. Let $\ep'>0$ be \textit{s.t.}\ $(2CT+1)C\ep'<\frac{\ep}{2}$ and $\delta >0$ the parameter associated with the small-state STLC of \eqref{schr} for $\ep'$ and $T$. Let $\psi_f\in H^3_{(0)}(0,1)\cap \mathcal{S}$ with $\left\|\psi_f-\psi_1(T)\right\|_{H^3} <\delta$ and $u\in L^2((0,T),\rr^r)$ the associated control. In particular, $u_1\in\mathcal{C}^0([0,1],\rr^r)$. We now prove that $\left\| u_1\right\|_{L^{\infty}(0,T)}\leq\ep$.
		By contradiction, we assume that this is not the case and we define 	$$\kappa:=\sup\{t\in [0,T]; \ \forall s\in [0,t], \ \left|u_1(s)\right|\leq \ep\}<T.$$
		As $u_1(0)=0$, one has $\kappa>0.$ By \eqref{small-state1}, one has, for all $t\in [0,\kappa[$,
		$$\int_0^t\left|u_1(s)\right|^2\mathrm{d}s\leq 2C^2\left(\frac{\kappa^2}{2}\left\|u_1\right\|_{L^2(0,t)}^2+\int_0^t\left\|\psi(s;u,\ph_1)-\psi_1(s)\right\|_{L^2(0,1)}^2\mathrm{d}s\right).$$
		Using the definition of small-state STLC and $(CT)^2\leq \frac{1}{2}$, one gets, for all $t\in [0,\kappa[$,
		$$\left\|u_1\right\|_{L^2(0,t)}\leq 2C\ep'\sqrt{T}.$$
		Using this estimate in \eqref{small-state1}, we finally get, for all $t\in [0,\kappa[$,
		$$\left\|u_1\right\|_{L^{\infty}(0,t)}\leq (2CT+1)C\ep'<\frac{\ep}{2},$$
		thanks to the definition of $\ep'$.	Taking the limit $t\to\kappa^-$, we obtain a contradiction. Consequently, $\left\|u_1\right\|_{L^{\infty}(0,T)}\leq \ep$ and \eqref{schr} is $W^{-1,\infty}$-STLC in $H^3_{(0)}(0,1)$ around the ground state.
		
		The indirect implication is classical and follows from the known -- see for example \cite[Proposition $4.9$]{bournissou2021quad} -- estimate \eqref{remainderquadu0}.
	\end{proof}
	\section*{Acknowledgement}
	The author would like to sincerely thank Karine Beauchard and Frédéric Marbach for the many discussions that shaped this article. The author acknowledges support from grants ANR-20-CE40-0009 (Project TRECOS) and ANR-11-LABX-0020 (Labex Lebesgue), as well as from the Fondation Simone et Cino Del Duca – Institut de France. We would like to acknowledge the assistance of volunteers in putting
	together this example manuscript and supplement.

\printbibliography
\end{document}